\documentclass[10pt, reqno]{amsart}
\usepackage{graphicx, amssymb, amsmath, amsthm}
\numberwithin{equation}{section}
\usepackage{geometry}
\usepackage{amscd}

\let\Re=\undefined\DeclareMathOperator*{\Re}{Re}

\newcommand{\R}{\mathbb{R}}

\newcommand{\Z}{\mathbb{Z}}

\newcommand{\N}{\mathcal{N}}
\usepackage{enumitem}
\newtheorem{theorem}{Theorem}[section]

\newtheorem{lemma}[theorem]{Lemma}

\newtheorem{corollary}[theorem]{Corollary}

\newtheorem{proposition}[theorem]{Proposition}

\theoremstyle{definition}
\newtheorem{definition}[theorem]{Definition}
\newtheorem{remark}[theorem]{Remark}

\makeatletter
\newcommand{\Extend}[5]{\ext@arrow0099{\arrowfill@#1#2#3}{#4}{#5}}
\makeatother

\begin{document}
\title[Nonlinear Schr\"{o}dinger equations]{ \bf Global well-posedness for rough solutions of  defocusing cubic NLS on three dimensional compact manifolds  }

	\author{Qionglei Chen}
	\address{Qionglei Chen
	 \newline \indent Institute of Applied Physics and Computational Mathematics,
Beijing, 100088, China.
\newline\indent
National Key Laboratory of Computational Physics, Beijing 100088, China}
\email{chen\_qionglei@iapcm.ac.cn}

\author{Yilin Song}%
	\address{Yilin Song
\newline \indent The Graduate School of China Academy of Engineering Physics,
     Beijing 100088,\ P. R. China}
\email{songyilin21@gscaep.ac.cn}

\author{Jiqiang Zheng}
\address{Jiqiang Zheng
\newline \indent Institute of Applied Physics and Computational Mathematics,
Beijing, 100088, China.
\newline\indent
National Key Laboratory of Computational Physics, Beijing 100088, China}
\email{zheng\_jiqiang@iapcm.ac.cn, zhengjiqiang@gmail.com}

\begin{abstract}
In this article, we investigate the global well-posedness for cubic nonlinear Schr\"{o}dinger equation(NLS) $  i\partial_tu+\Delta_gu=\left|u\right|^2u$ posed on the three dimensional compact manifolds $(M,g)$ with  initial data $u_0\in H^s(M)$ where $s>\frac{\sqrt{21}-1}{4}$ for Zoll manifold and $s>\frac{1+3\sqrt{5}}{8}$ for the product of spheres $\Bbb{S}^2\times\Bbb{S}^1$. We utilize the multilinear eigenfunction estimate on compact manifold to treat the   interaction of different frequencies, which is more complicated compared to the case of flat torus \cite{Fan} and waveguide manifold \cite{Zhao-Zheng}. Moreover, combining with
the  I-method adapted to the non-periodic case, bilinear Strichartz estimates  along with the scale-invariant $L^p$ linear Strichartz estimates, we partially obtain the similar result of \cite{Zhao-Zheng} on non-flat compact manifold setting.    As a consequence, we obtain the polynomial bounds of the $H^s$ norm of solution $u$.
\end{abstract}

 \maketitle

\begin{center}
 \begin{minipage}{100mm}
 \end{minipage}
 \end{center}


\section{Introduction}

\noindent

The defocusing cubic nonlinear Schr\"{o}dinger equation(NLS) posed on the three dimensional compact manifold $(M,g)$ is given by
\begin{align}\label{1}
  \begin{cases}
    i\partial_tu+\Delta_gu=\left|u\right|^2u, & (t,x)\in\Bbb{R}_t\times M, \\
    u(0,x)=u_0(x),
  \end{cases}
\end{align}
where $u(t,x)$ is a complex-valued function on $\Bbb{R}_t\times M$. We denote $\Delta_g$ by the standard Laplace-Beltrami operator on manifold. In the local coordinate, it can be written as
\begin{align*}
  \Delta_g=\sum_{i,j=1}^{3}\frac{1}{\sqrt{\left|g\right|}}\frac{\partial}{\partial x^j}\left(\sqrt{\left|g\right|}g^{ij}(x)\frac{\partial}{\partial x^i}\right)
\end{align*}
where $\left|g\right|$ denotes the determinant of metric $g$ and $g^{ij}=(g_{ij})^{-1}$. In this paper, we focus on the three-dimensional Zoll manifold and product of spheres $\Bbb{S}^2\times\Bbb{S}^1$.   Zoll manifolds are a class of compact manifold in which the geodesic flow has common periods, see \cite{Besse} for more backgrounds. The typical example is the unit sphere $\Bbb{S}^3$.

The nonlinear Schr\"{o}dinger equation $\eqref{1}$ can be seen as an infinite dimensoinal Hamiltonian system on $L^2(M)$ and its solutions enjoys the following two conservation laws,
\begin{align*}
\mbox{Mass: }&  M(u)=\int_{M}\left|u\right|^2dx=M(u_0),\\
\mbox{Energy: }&  E(u)=\frac{1}{2}\int_{M}\left|\nabla_gu\right|^2dx+\frac{1}{4}\int_{M}\left|u\right|^4dx=E(u_0),
\end{align*} 
where $\left|\nabla_gu\right|^2=g(\nabla u,\nabla u)=\sum\limits_{i,j=1}^{3}g^{ij}(x)\partial_iu(x)\partial_ju(x)$.

In the Euclidean setting, $\eqref{1}$  naturally arises in various physical contexts such as nonlinear optics and quantum mechanics (see \cite{Sulem-Sulem} for more physical backgrounds). The presence of a non-Euclidean metric would correspond to a medium with a variable optical index.
A similar equation to $\eqref{1}$ usually appears in the study of Bose-Einstein condensation, where a confining potential is incorporated into the classical Laplace operator. The effect of this confining potential can be akin to localization on a compact manifold. We refer to \cite{Gerard2,Maxwell} for  more  physical backrounds corresponding  to the nonlinear Schr\"odinger equation on compact manifolds.

\subsection{Recent progress of the local and global theory of NLS}
In this part, we will survey the related research in both the local and global theory of the nonlinear Schr\"{o}dinger equation(NLS) in flat and non-flat settings. 
\subsubsection{Local Theory}
The local well-posedness of $\eqref{1}$ is influenced by the choice of $M$. When $M=\Bbb{R}^3$, the equation $\eqref{1}$ is invariant under the scaling transform  
\begin{align*}
u(t,x)\rightarrow  u_\lambda(t,x)=\lambda u(\lambda^2t,\lambda x),\quad\forall\lambda>0.
\end{align*} 
Indeed, one can check that $\dot{H}^{\frac{1}{2}}$ is invariant under the above scaling transform. 
Strichartz estimates allow us to derive the local existence of $\eqref{1}$ in $H^s(\Bbb{R}^3)$ for $s\geqslant \frac{1}{2}$, see Cazenave\cite{Caze} for more details about the local theory of NLS on $\Bbb{R}^3$.
  
  The seminal work on local well-posedness for the NLS on compact manifolds without boundary was pioneered by Bourgain\cite{Bour}. In that paper, he initially established the periodic Strichartz estimates for the linear Schr\"{o}dinger equation by the analytical number theory. The Strichartz estimates is  optimal when  the dimensions $d=1,2$. However, when $M=\Bbb{T}^3$, Bourgain\cite{Bour-math} only obtain the local well-posedness in $H^s$ for $s>\frac{2}{3}$  which is significantly above the optimal regularity threshold. Recently, Bourgain and Demeter\cite{BD} proved the  scale-invariant Strichartz estimates on both rational and irrational tori, 
 \begin{align*}
   \big\Vert e^{it\Delta_{\Bbb{T}^3}}\phi\big\Vert_{L_{t,x}^p([0,1]\times\Bbb{T}^3)}\lesssim N^{\frac{3}{2}-\frac{5}{p}+\varepsilon}\Vert \phi\Vert_{L^2(\Bbb{T}^3)},\forall\; p\geqslant\frac{10}{3},
 \end{align*}
 where the initial data $\phi\in L^2(\Bbb{T}^3)$ satisfying $supp(\hat{\phi})\subset[-N,N]^3$. Utilizing these refined Strichartz estimates, one can obtain local well-posedness  in $H^s$ for $s>\frac{1}{2}$. Later, Killip-Visan\cite{KV} extended the result to rectangular  tori. 
  The $H^\frac{1}{2}$ critical local well-posedness on rectangular tori  was  demonstrated by Lee\cite{Lee}, where he employs the Bony paraproduct technique and the refined bilinear estimates in $U^p$-$V^p$ type atomic space.

For a general compact manifold without boundary, Burq-G\'{e}rard-Tzvetkov\cite{Burq1} developed the semiclassical analysis argument to construct the parametrix of linear Schr\"{o}dinger equation and subsequently prove the Strichartz estimates with loss of regularity,
\begin{align*}
  \Vert e^{it\Delta_g}f\Vert_{L_t^q(I,L_x^r(M))}\leqslant C(I)\Vert f\Vert_{H^\frac{1}{q}},\quad\frac{2}{q}=3\left(\frac{1}{2}-\frac{1}{r}\right),\hspace{1ex}q\geqslant2,r<\infty.
\end{align*} As a direct consequence, they proved the local well-posedness for $\eqref{1}$ on arbitrary three dimensional closed manifold in $H^s(M)$ with $s>1$. However, it is noted that the case 
$s=1$ remains unresolved. When $M$ is a Zoll manifold, Burq-G\'{e}rard-Tzvetkov\cite{Burq3} showed the local well-posedness for $\eqref{1}$ in $H^s$ with $s>\frac{1}{2}$ by establishing the  bilinear  estimates in Bourgain spaces. In \cite{Burq3}, they also showed the local well-posedness for $\eqref{1}$  in $H^s(\Bbb{S}^2\times\Bbb{S}^1)$ with $s>\frac{3}{4}$.  
\subsubsection{Global theory}
In this article, we focus on the $H^1$ subcritical regime. For $u_0\in H^1$, the global well-posedness in energy space can be established through the utilization of  the mass-energy conservation laws. Hence, our primary interest lies in the global well-posedness of $\eqref{1}$ below the energy space.

The initial breakthrough in this area was achieved through the high-low Fourier truncation method pioneered in Bourgain \cite{Bour2} in which he employed refined bilinear Strichartz inequalities, frequency decomposition, and perturbation techniques to demonstrate that the cubic NLS on $\Bbb{R}^3$ is globally well-posed in $H^s$ with $s>\frac{11}{13}$ such that 
\begin{align*}
  u(t)-e^{it\Delta}u_0\in H^1(\Bbb{R}^3).
\end{align*} 
Building on the work of \cite{Bour2}, Colliander-Keel-Staffilani-Takaoka-Tao(I-team)\cite{CKSTT01,CKSTT2002,CKSTT2004} developed the I-method  to study the low regularity issues for a wide class of nonlinear dispersive equations. In comparison to Bourgain’s result \cite{Bour2},  I-team achieved  both global well-posedness and scattering in $H^s(\Bbb{R}^3)$ with $s>\frac{4}{5}$ by taking advantage of the I-method along with interaction Morawetz estimates from  \cite{CKSTT2004}. Dodson\cite{Dodson1} extended those results to $s>\frac{5}{7}$ by means of linear-nonlinear decomposition and then Su\cite{Su} further improved the result to $s>\frac{49}{74}$. For the radial initial data, Dodson\cite{Dodson2} utilized the long-time Strichartz estimates to prove the  global well-posedness in $H^s$ for $s>\frac{1}{2}$. We refer to \cite{M-W-Z,M-Z} for related low regularity problem on Euclidean spaces. Recently, Dodson\cite{Dodson-L^p} proved the global well-posedness for initial data $u_0\in W^{\frac{7}{6},\frac{11}{7}}(\Bbb{R}^3)\subset\dot{H}^\frac{1}{2}(\Bbb
{R}^3)$ by the perturbation method. 
 We also refer the readers to \cite{Kenig-Merle,Murphy} for the critical norm conjecture. 

When $M$ is a compact manifold, Bourgain made the first breakthrough in this direction. In \cite{Bour4}, he utilized the normal-form reduction,  I-method  and a refined trilinear Strichartz estimates to establish the global well-posedness for quintic NLS in $H^s(\Bbb{T})$ with $s>\frac{1}{2}-$.  Later, De Silva-Pavlovi\'{c}-Staffilani-Tzirakis\cite{T^2} proved the global well-posedness of the cubic NLS in $H^s(\Bbb{T}^2)$ with $s>\frac{2}{3}$ by using I-method in conjunction with refined bilinear Strichartz estimates
\begin{align}\label{bil-1}
  \bigg\Vert \eta(t)U_\lambda\phi_1\cdot\eta(t)U_\lambda\phi_2\bigg\Vert_{L_{t,x}^2}\lesssim N_2^\varepsilon\left(\frac{1}{\lambda}+\frac{N_2}{N_1}\right)^\frac{1}{2}\Vert\phi_1\Vert_{L^2}\Vert\phi_2\Vert_{L^2},
\end{align} 
where $1\leqslant N_2\ll N_1$, $\lambda\geqslant1$, $\eta(t)\in C_0^\infty([-4,4])$, $\phi_j\in L^2(\lambda\Bbb{T}^2)$ is spectrally localized at $\left|k_j\right|\sim \lambda N_j$($j=1,2$) and 
\begin{align*}
  U_\lambda(t)u_0(x)=\frac{1}{\lambda}\sum_{k\in\frac{1}{\lambda}\Bbb{Z}^2}e^{2\pi ik\cdot x-4\pi it\left|k\right|^2}\hat{u}_0(k).
\end{align*} 
We note that when $\lambda\rightarrow\infty$, this bilinear estimates $\eqref{bil-1}$ consistent with the bilinear Strichartz estimates in $\Bbb{R}^2$ up to the $\varepsilon$ loss. It can be further improved to $s>\frac{3}{5}$ by using the argument in a recent paper \cite{Deng}.

For the two-dimensional general compact manifold, Zhong\cite{Zhong} proved the global well-posedness for cubic NLS in $H^s$ with $s>\frac{\sqrt{2}}{2}+(1-\frac{\sqrt{2}}{2})s_0$ under  the following bilinear Strichartz assumption,
\begin{align}\label{bil-assum}
  \big\Vert e^{it\Delta_g}fe^{it\Delta_g}g\big\Vert_{L_{t,x}^2}\lesssim N_2^{s_0}\Vert f\Vert_{L^2}\Vert g\Vert_{L^2},
\end{align}
where
\begin{align*}
  \mathbf{1}_{N_1\leqslant\sqrt{-\Delta}<2N_1}(f)=f,\quad  \mathbf{1}_{N_2\leqslant\sqrt{-\Delta}<2N_2}(g)=g\mbox{ with }N_2\leqslant N_1.
\end{align*}
When $M=\Bbb{S}^2$, Burq-G\'{e}rard-Tzvetkov\cite{Burq2} proved $\eqref{bil-assum}$ holds for $s_0=\frac{1}{4}$ and then cubic NLS on $\Bbb{S}^2$ is globally well-posed for $s>\frac{2+3\sqrt{2}}{8}$.
  Hani\cite{Hani1,Hani2}  proved the refined bilinear oscillatory integral estimate and then used it to establish the bilinear Strichartz estimates with frequency-localized initial data on $2$-dimensional closed manifolds. It is worth noting that due to the compactness of the geometry, this improvement of the bilinear Strichartz estimates holds only within a very short time interval depending on the frequency-size of the initial data. To establish bilinear estimates over a longer time interval $[0,1]$, he employed the scaling argument for the time variable and then obtained the bilinear estimates on the space-time slab $[0,1]\times M_\lambda$, where
$M_\lambda$ denotes the rescaled manifold of $M$,
\begin{align}\label{bil-2}
  \bigg\Vert e^{it\Delta_\lambda}\Delta_{N_1}u_0e^{it\Delta_\lambda}\Delta_{N_2}u_1\bigg\Vert_{L_{t,x}^2([0,1]\times M_\lambda)}\lesssim \bigg(\big(\frac{N_2}{N_1}\big)^\frac{1}{2}+\big(\frac{N_2}{\lambda}\big)^\frac{1}{2}\bigg)\Vert u_0\Vert_{L^2(M_\lambda)}\Vert u_1\Vert_{L^2(M_\lambda)},
\end{align}
where $\Delta_\lambda$ denotes the Laplace-Beltrami operator on $M_\lambda$ and $\Delta_N=\sum_{\langle\lambda_k\rangle\sim N}\pi_k$ where  $\lambda_k$ denotes the $k$-th eigenvalues of $\sqrt{-\Delta_{\lambda}}$ and $\pi_k$ denotes the spectral projector associated with $\lambda_k$.  
Subsequently, by combining with the I-method along with the multilinear eigenfunction estimate, he showed the global well-posedness of cubic NLS on $2$-dimensional closed manifold in $H^s$ for $s>\frac{2}{3}$.

 Recently, Deng-Fan-Yang-Zhao-Zheng\cite{Deng} successfully established the refined bilinear Strichartz estimates on rescaled waveguide $\Bbb{R}\times\Bbb{T}_{\lambda}$, thereby eliminating  the $\varepsilon$-loss of $\eqref{bil-1}$. Based on this achievement, they demonstrated the global well-posedness of the cubic NLS  on $\Bbb{R}\times\Bbb{T}$ for initial data $u_0\in H^s$ with $s>\frac{3}{5}$.  The key improvement  stems from the newly established almost conservation law for a modified energy  $E(Iu)$, which depends on $\lambda\sim N^\alpha$ with $\alpha\sim\frac{1-s}{s}$,
 \begin{align*}
   \left|E(Iu(\delta))-E(Iu(0))\right|\lesssim\frac{1}{\lambda^\frac{1}{2} N^{1-}}\sim\frac{1}{N^{1+\frac{\alpha}{2}-}}.
 \end{align*}

For the three-dimensional manifold, the first result is given by Zhao-Zheng\cite{Zhao-Zheng}.  In \cite{Zhao-Zheng}, they employed the bilinear decoupling theorem\cite{Fan} to establish the refined bilinear Strichartz estimates on the wave-guide manifold $\Bbb{R}^{m_1}\times \Bbb{T}^{m_2}$ where $m_1+m_2=3$. As an application, they showed that the solution $u$ of $\eqref{1}$ on $\Bbb{R}^{m_1}\times \Bbb{T}^{m_2}$ is globally well-posed in $H^s$ with $s>\frac{5}{6}$.  The novel contribution of \cite{Zhao-Zheng} lies in the development of refined bilinear Strichartz estimates on a rescaled manifold 
 $\Bbb{R}^{m_1}\times\Bbb{T}_\lambda^{m_2}$ where $\lambda\in[1,N]$ and $N\in2^\Bbb{N}$,
\begin{align}\label{T^3}
  \big\Vert U_\lambda(t)u_{N_1}U_\lambda(t)v_{N_2}\big\Vert_{L_{t,x}^2([0,1]\times\Bbb{R}^{m_1}\times\Bbb{T}^{m_2})}\lesssim N_2^\varepsilon\left(\frac{1}{\lambda}+\frac{N_2^2}{N_1}\right)^\frac{1}{2}\Vert u_{N_1}\Vert_{L^2}\Vert v_{N_2}\Vert_{L^2}.
\end{align} 
 
In this article, we are interested in demonstrating the global well-posedness of Cauchy problem $\eqref{1}$ on three-dimensional compact manifold    with initial data $u_0\in H^s$ for some $s\in(0,1)$.
   
To the authors' knowledge, there have been no prior results in the literature concerning the deterministic global well-posedness of the cubic nonlinear Schr\"{o}dinger equation (NLS)  below the energy space on a three-dimensional compact manifold that is not flat. Specifically, we focus on two distinct cases: Zoll manifold and  product of spheres, $\Bbb{S}^2\times\Bbb{S}^1$.
\subsection{Main Result}
Our main theorem in this paper is as follows.
\begin{theorem}\label{Thm1}Let $s\in(0,1)$ and $u_0\in H^s(M)$.
\begin{enumerate}
\item Let $M$ be a Zoll manifold and $s>\frac{\sqrt{21}-1}{4}$.  The initial value problem $\eqref{1}$ is globally well-posed for $u_0\in H^s$. Moreover, we establish the polynomial bound of the global solution 
\begin{align}
  \Vert u(T)\Vert_{H^s(M)}\lesssim T^{\frac{3}{2p(Zoll)}(1-s)+},\quad \forall\; T\gg1,
\end{align}
where $p(Zoll)=-\frac{3(1-s)}{2s-1}+s-\frac{1}{2}$.
\item Let  $M$ be the product of spheres and $s>\frac{1+3\sqrt{5}}{8}$. The initial value problem $\eqref{1}$ is globally well-posed for $u_0\in H^s$. Moreover, we have
\begin{align}
  \Vert u(T)\Vert_{H^s(M)}\lesssim T^{\frac{3}{2p(\Bbb{S}^2\times\Bbb{S}^1)}(1-s)+},
\end{align}
where $p(\Bbb{S}^2\times\Bbb{S}^1)=-\frac{5(1-s)}{4s-3}+s-\frac{3}{4}$.
\end{enumerate}
\end{theorem}
\begin{remark}
	The second result of Theorem \ref{Thm1} can be generalized to the more general manifold such as the product of two-dimensional Zoll manifold and the circle. 
\end{remark}
\begin{remark}
	In the case of compact manifold setting, it is difficult to apply the harmonic analysis tools as well as the analytic number theory to establish the bilinear Strichartz estimates on torus \cite{Fan} or waveguide manifold \cite{Zhao-Zheng}. If one can prove the bilinear estimates similar to \eqref{T^3}, we can show the global well-posedness for $s>\frac{5}{6}$.
\end{remark}

\begin{remark}
 We note that  using the upside-down I-operator along with the argument in \cite{Deng-Germain,Zhao-Zheng}, it is possible to obtain the growth of higher-order Sobolev norms for both Zoll manifold and product of spheres which is tightly linked to the weak wave turbulence theory. 
\end{remark}

\begin{remark}
  The constraints of $s$ obtained in this paper showed the difference between Zoll manifold and product of spheres. The geometry of manifold plays the important role in establishing the bilinear  Strichartz estimates. More precisely, the bilinear estimates is determined by the distributions of eigenvalues associated with Laplace-Beltrami operator. By the direct calculus, we can see that the distribution of eigenvalues assocaited with $-\Delta_g$ on $\Bbb{S}^2\times\Bbb{S}^1$ is more concentrated than the case  of Zoll manifold.  This is the reason why the local and global well-posedness for $\eqref{1}$ obtained in \cite{Burq3} and our paper on product of spheres is worse than Zoll manifold, see Section \ref{sec:bilstriest} for more details. 
\end{remark}
\begin{remark}
	As mentioned above, our result give the deterministic global well-posedness in the low regularity Sobolev spaces. For the probablistic well-posedness theory, we refer to Burq-Thomann-Tzvetkov\cite{Burq-gibbs}. For the sub-quintic nonlinearities, they defined a Gibbs measure  with support in $H^\sigma$ with $\sigma<\frac{1}{2}$ and to construct the global solution in these Sobolev spaces where the initial data was restricted to the zonal functions.  
\end{remark}

Let us give some comments on our proof of the main result. The proof of the main theorem is based on the I-method adapting to the compact manifold setting. The $I$-method consists in smoothing out the $H^s$ data for $0<s<1$ to access the local and global theory in the $H^1$-level. To achieve this, we define the I-operator as a spectral multiplier
\begin{align*}
  Iu=\sum_{n_k}m(n_k)\pi_{k}u,\quad\forall u\in H^s(M),
\end{align*}
where $n_k$ denotes the $k$-th eigenvalues associated with $\sqrt{-\Delta_g}$ and $\pi_k$ denotes the $L^2$ spectral projection operator and $m(\xi)$ is the smooth cut-off function such that
\begin{align*}
  m(\xi)=\begin{cases}
           1, & \mbox{if } |\xi|\leqslant N, \\
           \big(\frac{N}{|\xi|}\big)^{1-s}, & \mbox{if }|\xi|>2N.
         \end{cases}
\end{align*}
 Thus, $I$ is an identical operator when $|\xi|\leqslant N$ and behaves like a smooth operator of order $1-s$ on high frequencies. By acting this operator on the both side of $\eqref{1}$, we obtain  $Iu$ as a solution of modified equation $\eqref{I}$
 \begin{align*}
   i\partial_tv+\Delta_g v=I(\left|u\right|^2u)=I(\left|I^{-1}v\right|^2I^{-1}v),
 \end{align*}
 where $v:=Iu$.
 
  By the standard argument, we can divide the proof into three steps.
 
The first step in using the I-method machinery is to establish the good local well-posedness for the modified equation $Iu$.   
In \cite{Hani2}, they used the scaling argument to ensure the lifespan of local solutions to $Iu$ is larger than $1$. However, their method cannot work good in the three-dimensional case. We note that following the idea in Hani\cite{Hani2} to handle the bilinear Strichartz estimates on three-dimensional manifold, one can only obtain the bilinear estimates with $s_0>1$. For the 3D periodic setting, the decoupling argument can capture the $\lambda$ gain in establishing the bilinear estimates. This improvement  highly relies on the geometry of the torus and cannot hold on the general compact manifold.    This is the main difference between our setting and periodic/semi-periodic setting.  For the general three-dimensional compact manifold, Burq-G\'erard-Tzvetkov\cite{Burq1} only proved the local well-posedness for $s>1$ which misses the important endpoint.    Burq-G\'{e}rard-Tzvetkov\cite{Burq3}  further proved the improved local theory in $H^s$ with $s>\frac{1}{2}$ on Zoll manifolds and $s>\frac{3}{4}$ on  product of spheres $\Bbb{S}^2\times\Bbb{S}^1$.   Inspired by their work, we utilize the bilinear estimates established in \cite{Burq2} to    establish the following nonlinear estimates for $Iu\in X^{1,b_0}$ where $0<b_0<\frac{1}{2}$ which depends on the choice of $M$,
\begin{align*}
\Vert I(\left|u\right|^2u)\Vert_{X^{1,-b_0}([-\delta,\delta]\times M)}\lesssim \Vert Iu\Vert_{X^{1,b_0}([-\delta,\delta]\times M)}^3.
\end{align*}
 Then using the Sobolev embedding to obtain the $\delta^{\frac{3}{2}-b_0-}$-gain in time variable.  This can make us obtain the much longer lifespan of time-interval, which can reduce the numbers of iterating, which is crucial in running I-method. 
  
The second step in our proof is estimating the energy increment. Since $Iu$ is not a solution to $\eqref{1}$, thus the energy functional $E(Iu(t))$ is not conserved any more. Thus to iterate the solution from $[0,\delta]$ to $[0,T]$, we first need to estimate the energy increment on the time interval that the local solution exists. By the direct calculation, we can divide the energy increment into two main terms:
\begin{align*}
E(Iu(t))-E(Iu(0))\stackrel{\triangle}{=}\Re(iI_1-iI_2),
\end{align*}
where
\begin{gather*}
I_1=\sum_{n_i}\int_{0}^{\delta}\int_{M}\left[1-\frac{m(n_1)}{m(n_2)m(n_3)m(n_4)}\right]\overline{\pi_{n_1}(\Delta Iu)}\pi_{n_2}(Iu)\overline{\pi_{n_3}(Iu)}\pi_{n_4}(Iu)dxdt,\\
I_2=\sum_{n_i}\int_{0}^{\delta}\int_{M}\left[1-\frac{m(n_1)}{m(n_2)m(n_3)m(n_4)}\right]\overline{\pi_{n_1}(I(\left|u\right|^2u))}\pi_{n_2}(Iu)\overline{\pi_{n_3}(Iu)}\pi_{n_4}(Iu)dxdt.
\end{gather*}

In this following, we  emphasize the strategy we use to deal with the most  difficulty terms in $I_1$. The first trouble in estimating the term $I_1$ is the lack of Fourier transform and convolution theorem. In the Euclidean and periodic setting, the convolution theorem translates the Fourier transform of a product into the convolution of Fourier transforms. Combining with the localization argument, we can express the nonlinear term  $\left|u\right|^2u$ as the sums of the interaction between  frequency-localized functions $u_1,u_2,u_3$ centered at $\xi_i$. The convolution implies that the resonant condition $\xi_1-\xi_2+\xi_3=\xi$  which permits the exclusion of excessively high frequencies. In the non-periodic setting, the presence of highest frequency complicates the treatment of these terms, posing a significant obstacle to summing the localized nonlinear interactions and back to the original nonlinearity $\left|u\right|^2u$. To overcome this problem, we invoke the geometric lemma established in \cite{Hani2}, which provides precise control over the multilinear interactions of distinct frequencies. Additionally, we employ the spectral multiplier lemma from \cite{Hani2} as a discrete analogue to the Coifman-Meyer theorem. These tools enable us to reduce the estimate of $I_1$ to  the following integral,
 \begin{align*}
  \int_{0}^{t}\int_{M}f_1(t,x)f_2(t,x)f_3(t,x)f_4(t,x)dxdt
 \end{align*}
 where $f_1,\cdots,f_4$ are spectrally localized functions.
We employ bilinear estimates in the space  $X^{s,b_0}$ rather than $X^{s,\frac{1}{2}+}$. This strategic choice allows us to capture a $\delta^{s_0}$-gain in time variable, where $s_0=\frac{1}{4}-$ for Zoll manifold and $s_0=\frac{1}{6}-$ for $\Bbb{S}^2\times\Bbb{S}^1$. On the other hand,  We use the almost orthogonal property to estimate $I_2$ where the extremely high frequency occurs. 
 For the case that the maximum frequency is comparable with  the medium frequency,  we use the scale-invariant $L^p$ to ensure the summability of all frequencies. 
  
To finish the proof of the main theorem, we iterate the local solutions($t\in[0,\delta]$) several times to time interval $[0,T]$ for arbitrary large $T\gg1$. By utilizing the Gagliardo-Nirenberg inequality and maintaining uniformly control  the energy  $E(Iu(k\delta))$ by $E(Iu(0))\sim N^{3(1-s)}$.  Therefore, we can obtain the relationship of $T$ and $N$ and the restriction on $s$. Consequently, the polynomial bound of $H^s$ norms can be derived.

\textbf{Organization of this paper}\hspace{2ex}The paper is organized as follows. In Section \ref{sec:prelim}, we give the basic properties of eigenvalues of Laplace-Beltrami operator on compact manifolds,  function spaces and the definition of $I$-operator. In Section \ref{sec:bilstriest}, we give the $L_{t,x}^2$ bilinear Strichartz estimates and $L_{t,x}^p$ Strichartz estimates on Zoll manifolds and product of spheres.  In Section \ref{sec:prothm}, we prove Theorem \ref{Thm1} via the I-method.
\subsection{Notations}
In this paper, we use $A\lesssim B$ to mean that there exists a constant such that $A\leqslant CB$, where the constant is not depending on $B$. We will also use $s+$ or $s-$, which means that there exists a small positive number $\varepsilon$ such that it is equal to $s+\varepsilon$ or $s-\varepsilon$ respectively. Denote $\langle x\rangle$ by $(1+\left|x\right|^2)^\frac{1}{2}$.




\section{Preliminaries}\label{sec:prelim}
\subsection{Function Space} 
In this part, we give definition of function spaces on compact manifold  which will be used in the proof. 

 We first define the Lebesgue norm of funcion on compact manifold. Denote $g(x)$ the canonical metric of $M$ and $dg$ the volume form associated with the Riemannian metric $g$, 
	\begin{align*}
		\Vert f\Vert_{L^p(M)}=\big(\int_{M}\left|f(x)\right|^pdg\big)^\frac{1}{p}.
	\end{align*}
	The Sobolev space $H^s(M)$ and Bourgain space on compact manifold are   $L^2$-based which their norms  can be defined by the spectral resolution. Therefore, we need to introduce the properties of  eigenvalues and eigenfunctions  of $\sqrt{-\Delta_g}$.  
	Let $n_k$ be the $k$-th eigenvalue of $\sqrt{-\Delta_g}$ with multiplicity. Denote that $e_k$ the eigenfunctions of the $\sqrt{-\Delta_g}$ associated with $n_k$ which form an orthogonal eigenfunction basis of the $L^2(M)$. For a given $f\in L^2(M)$, we have
	\begin{align*}
		f=\sum_{\ell=1}^{\infty}\pi_\ell f=\sum_{\ell=1}^{\infty}\langle f,e_\ell\rangle e_\ell, \quad-\Delta_{g}e_\ell=n_\ell^2e_\ell,
	\end{align*}
	where $\pi_\ell$ denote the spectral projector.
	We define the Sobolev space($s\geqslant0$) as follows,
	\begin{align*}
		H^s(M)=\left\{u\in H^s(M):\Vert u\Vert_{H^s(M)}^2=\sum_{\ell=1}^{\infty}\langle n_\ell\rangle^{2s}\Vert \pi_\ell u\Vert_{L^2(M)}^2<\infty\right\}.
	\end{align*}
	
	By the spectral representation, the linear Schr\"{o}dinger flow can be rewritten as
	\begin{align}\label{spectral-represen}
		e^{it\Delta_g}f=\sum_{k=1}^{\infty}e^{-itn_k^{2}}\pi_kf,\hspace{1ex}\forall f\in L^2(M).
	\end{align}
	 We denote $\Delta_N$ as $$\Delta_N=\sum\limits_{N\leqslant \langle n_k\rangle<2N}\pi_k.$$

We now state the Bernstein inequality on  compact manifold, which was proved in  \cite[Lemma 3.1]{Burq3}.
\begin{lemma}[Bernstein's inequality] Let $M$ be a three-dimensional compact manifold,
there exists a constant $C>0$ such that for every $q\in[2,\infty]$, every $u\in L^2(M)$,
\begin{align*}
  \big\Vert \Delta_N(u)\big\Vert_{L^q(M)}\leqslant CN^{\frac{3}{2}-\frac{3}{q}}\big\Vert\Delta_{N}(u)\big\Vert_{L^2(M)}.
\end{align*}
\end{lemma}

	Using the spectral, representation $\eqref{spectral-represen}$, we define the space-time conormal space which is called the Bourgain space.
	\begin{definition}[$X^{s,b}$ space]
		Let $s\geqslant0$ and $b\in\Bbb{R}$,
		\begin{align*}
			X^{s,b}(\Bbb{R}\times M)=\left\{u\in\mathcal{ S}(\Bbb{R},L^2(M)):\Vert u\Vert_{X^{s,b}(\Bbb{R}\times M)}<\infty\right\},
		\end{align*}
		where
		\begin{align*}
			\Vert u\Vert_{X^{s,b}(\Bbb{R}\times M)}^2=\Vert e^{-it\Delta_g}u\Vert_{H_t^b(\Bbb{R},H^s(M))}^2=\sum_{k=1}^{\infty}\langle n_k\rangle^{2s}\Vert\langle\tau+n_k^{2}\rangle^b\widehat{\pi_kf}(\tau)\Vert_{L_{\tau,x}^2}^2,
		\end{align*}
		where  $\widehat{\pi_kf}(\tau)$ denote the Fourier transform  with respect to the time variable.
		Moreover, if $b=\infty$, then $u\in X^{s,\infty}=\bigcap\limits_{b\in\Bbb{R}}X^{s,b}$. For $s\leqslant0$, we can define the negative Bourgain space by duality.
		
		For $T>0$, we can define the restricted Bourgain space equipped with the norm
		\begin{align*}
			\Vert u\Vert_{X_T^{s,b}([-T,T]\times M)}=\inf\left\{\Vert w\Vert_{X^{s,b}(\Bbb{R}\times M)}<\infty,\hspace{1ex}w|_{[-T,T]\times M}=u\right\}.
		\end{align*}
	\end{definition}

	\begin{lemma}[Some basic properties of $X^{s,b}$,\cite{Tao-book}]\label{X^{s,b}}
		\begin{enumerate}
			\item For $s\geqslant0$ and $b>\frac{1}{2}$, we have the Sobolev embedding $X^{s,b}(\Bbb{R}\times M)\hookrightarrow C(\Bbb{R},H^s(M))$.
			\item For $s'\leqslant s$ and $b'\leqslant b$, we have the including property
			\begin{align*}
				X_T^{s,b}\subset X_{T}^{s',b'}.
			\end{align*}
\item The following estimate holds
\begin{align}\label{useful}
  X^{0,\frac{1}{4}}(\Bbb{R}\times M)\hookrightarrow L^4(\Bbb{R}\times L^2(M)).
\end{align} 
		\end{enumerate}
	\end{lemma}
The proof of Lemma \ref{X^{s,b}} can be found in the textbook of Tao\cite{Tao-book}.
	
Next we state the  homogeneous and inhomogeneous estimates for solutions to the linear Schr\"{o}dinger equations in Bourgain spaces, see Tao\cite{Tao-book} for the details.
	\begin{lemma}[Homogeneous estimate,\cite{Tao-book}]
		Let $s,b>0$ and $f\in H^s(M)$, then there exists a constant $C>0$ such that
		\begin{align*}
			\Vert \psi(t)e^{it\Delta_g}f\Vert_{X_T^{s,b}(\Bbb{R}\times M) }\leqslant C\Vert f\Vert_{H^s(M)}.
		\end{align*}
	\end{lemma}
	
	\begin{lemma}[Inhomogeneous estimate,\cite{Tao-book}]Let $0<b'<\frac{1}{2}$ and $0<b<1-b'$. Then for all $f\in X_T^{s,-b'}([-T,T]\times M)$, the inhomogeneous term $w(t,x)=\int_{0}^{t}e^{i(t-s)\Delta_g}f(s)ds$ satisfies the following estimate
		\begin{align*}
			\Vert \psi(t/T)w\Vert_{X_T^{s,b}([-T,T]\times M)}\leqslant CT^{1-b-b'}\Vert f\Vert_{X_T^{s,-b'}([-T,T]\times M)}.
		\end{align*}
	\end{lemma}

\subsection{Multilinear eigenfunction estimate and spectral multiplier lemma   }
In this subsection, we will review the result of multilinear eigenfunction estimate and multiplier lemma in \cite{Hani2}.
	
	The spectral resolution of product of two eigenfunctions is well understood in the case of torus and spheres. Denote $e_k$ the $k$-th order eigenfunction of the Laplace-Beltrami operator. For the case of torus $\mathbb{T}^d$, $e_k(x)=e^{ikx}$ enjoys algebraic properties,
	\begin{equation*}
		e_m(x)e_n(x)=e_{m+n}(x).
	\end{equation*}
	For $d$-dimensional sphere $\Bbb{S}^d$, the spherical harmonic functions enjoy similar properties. The product of two spherical harmonics of order $m$ and $n$ respectively can be written as the sum of   spherical harmonics of order less than $m+n$,
	\begin{equation*}
		e_m(x)e_n(x)=\sum_{0\le k\le m+n}a_ke_{k}(x).
	\end{equation*}
	We refer to \cite{Stein} for more details. If $n_1>n_2+n_3+n_4$, we have the almost orthogonal property
\begin{align*}
  \left|\int_{M}e_{n_1}(x)e_{n_2}(x)e_{n_3}(x)e_{n_4}(x)dx\right|=o(1).
\end{align*}
		For the general compact manifold, Burq-G\'{e}rard-Tzvetkov\cite{Burq2} proved the weaker almost orthogonal estimates  when $n_1>C(n_2+n_3+n_4)$. Yang\cite{Yang} generalized \cite{Burq2} to the multilinear case.
	\begin{lemma}[Almost orthogonal,\cite{Burq2,Yang}]\label{orthogonal}
		Let $e_{n_j}$ be eigenfunctions of the Laplacian associated with eigenvalues $n_j$, $j=1,\cdots,k$. There exists $C>0$ such that if $n_1\geqslant C(n_2+n_3+\cdots+n_k)$, $k\geqslant2$, then for every $\nu\in\Bbb{N}$ we have
		\begin{align}\label{almost}
			\left|\int_{M}e_{n_1}(x)e_{n_2}(x)\cdots e_{n_k}(x)dx\right|\leqslant n_1^{-\nu}\prod_{j=1}^{k}\Vert e_k\Vert_{L^2(M)}.
		\end{align}
	\end{lemma}
	The above lemma is crucial in establishing the local well-posedness of the energy-subcritical NLS in $X^{s,b}$. The proof of Lemma \ref{orthogonal} relies on the following half-wave operator approximation obtained by Sogge\cite{Sogge,Sogge2}
	\begin{equation*}
		\chi(\sqrt{-\Delta}-k)f=k^{\frac{d}{2}}T_kf+R_kf,
	\end{equation*}
	where $\chi\in C^{\infty}_0$ and $T_k$ denote the H\"{o}rmander-type oscillatory integral operators where the phase function is associated to the geodesic distance. Moreover, $R_{k}f$ denotes remainder term which satisfies
	\begin{equation*}
		\|R_{k}\|_{L^2\rightarrow H^\sigma}=O(k^{\sigma-N}),\quad \sigma=0,\cdots,N,\quad\forall N\in\Bbb{N}.
	\end{equation*}
	However, this estimate can not be used in dealing with the case $n_1\sim n_2$ where $n_1$ and $n_2$ are two high frequencies. More precisely, we need to establish the estimate when $n_1\sim n_2+C(n_3+\cdots+n_k)$. Inspired by examples of torus and spheres, we should expect that the product of eigenfunction is localized in the region $n_2+O(n_3)$ rather than $O(n_2)$. Using the technique in Riemannian geometry, Hani\cite{Hani2} gave the answer to this problem and we state the theorem of \cite{Hani2} in the following. 
	\begin{lemma}[Hani\cite{Hani2},Theorem 4.2]\label{keylemma}
		Let $M$ be a $d$-dimensional compact  manifold and $e_1,e_2,e_3, e_4$ be eigenfunctions of the Laplace-Beltrami operator corresponding to eigenvalues $n_1^2,\cdots,n_4^2$ respectively. Denote by $A_0$:
		\begin{align*}
			A_0=\int_{M}e_1(x)e_2(x)e_3(x)e_4(x)dx.
		\end{align*}
		Then, for any $n\in\Bbb{N}$:\begin{align*}
			A_0=\frac{(-2)^nA_n}{(n_1^2-n_2^2-n_3^2-n_4^2)^n}
		\end{align*}
		where $A_n$ is given by
		\begin{align*}
			A_n=\int_{M}e_1(x)(B_n(e_2,e_3,e_4)+C_n(e_2,e_3,e_4))(x)dx
		\end{align*}
		and $B_n(f_1,f_2,f_{3})$ and $C_n(f_1,f_2,f_{3})$ are trilinear operators of the form
		\begin{align*}
			&B_n(f_1,f_2,f_3)=\mathcal{O}_{\substack{i_1+i_2+i_3=2n\\0\leqslant i_1,i_3\leqslant2n}}(\nabla^{i_1}f_1*\nabla^{i_2}f_2*\nabla^{i_3}f_{3}),\\%
			&C_n(f_1,f_2,f_{3})=\mathcal{O}_{\substack{i_1+i_2+i_{3}=2n\\0\leqslant i_1,i_2,i_3\leqslant2n}}(\Tilde{R}_n*\nabla^{i_1}f_1*\nabla^{i_2}f_2*\nabla^{i_{3}}f_{3}),
		\end{align*}
		where $\nabla^if$ denotes the $i$-th covariant derivative of $f$ and for two tensors $A$ and $B$, $A*B$ denotes some contraction of $A\otimes B$ and $\mathcal{O}(A*B)$ denotes a linear combination of contractions of $A\otimes B$. $\Tilde{R}_n$ denotes a tensor obtained from the Riemannian curvature tensor by contracting and differentiating it a bounded number of times.
	\end{lemma}
	To estimate the energy increment in running the I-method, the key ingredient is the  Coifman-Meyer multiplier theorem,
\begin{align*}
	\Lambda(f_1,\cdots,f_k)&=\int_{\zeta_1+\cdots+\zeta_k=0}\overline{m}(\zeta_1,\cdots,\zeta_k)\widehat{f_1}(\zeta_1)\cdots\widehat{f_k}(\zeta_k)d\zeta_1\cdots d\zeta_k.
\end{align*}
It is of particular
interest to obtain  multilinear estimates involving $L^2$ and $X^{s,b}$ space.
Taking advantage of the Plancherel theorem and duality, we can prove these estimates by reducing it to
a weighted convolution integral in terms of the $L^2$ norms of the component functions.
Such analysis is valid if we replace $\Bbb{R}^d$ and $\Bbb{T}^d$ by another additive Abelian groups. Building on this fact, Tao\cite{Tao} developed the $[k;\Bbb{Z}]$-multiplier method in the study of dispersive equations.  Hani\cite{Hani2} generalized it to the  compact manifold setting where the group action is not an Abelian group. More precisely, he proved the  spectral multiplier lemma which is the discrete analogue of Coifman-Meyer multiplier theorem.
\begin{lemma}[\cite{Hani2},Lemma 5.1]\label{multiplier-lemma}
	Let $\Lambda$ be a $k$-linear multiplier associated with the multiplier $\Tilde{m}$ with the form
	\begin{align}\label{mult}
		\Lambda(f_1,\cdots,f_k)=\sum_{n_i}\Tilde{m}(n_1,\cdots,n_k)\int_{M}\pi_{n_1}f_1(x)\cdots\pi_{n_k}f_k(x)dx
	\end{align}
	and let $Y$ be a Banach space and $f_i$ satisfy the modulation stability property. More precisely, $\Tilde{f}_i=\sum_{n_i}e^{i\theta_in_i}\pi_{n_i}f_i(x)$ satisfy
	\begin{align*}
		\Vert \Tilde{f}_i\Vert_{Y}\lesssim\Vert f_i\Vert_{Y}.
	\end{align*}
	Assume that $\Tilde{m}$ satisfies the symbol-like estimate
	\begin{align*}
		\left|\partial_{\zeta_1}^{\alpha_1}\cdots\partial_{\zeta_k}^{\alpha_k}\Tilde{m}(\zeta_1,\cdots,\zeta_k)\right|\lesssim\langle\zeta_1\rangle^{-\alpha_1}\cdots\langle\zeta_k\rangle^{-\alpha_k}
	\end{align*}
	and the following estimate holds
	\begin{align}\label{Priori}
		\left|\int_{0}^{t}\int_{M}f_1\cdots f_kdxdt\right|\leqslant B\Vert f_1\Vert_{Y}\cdots\Vert f_k\Vert_{Y}
	\end{align}
	where $f_i=P_{N_i}f_i$ are spectrally localized to scales $N_i$. Then, there exists a constant $C>0$(depending only on $B$) such that
	\begin{align*}
		\left|\int_{0}^{t}\Lambda(f_1,\cdots,f_k)dt\right|\leqslant CB\Vert f_1\Vert_{Y}\cdots\Vert f_k\Vert_{Y}.
	\end{align*}
\end{lemma}

\begin{remark}
It is an interesting problem to establish the multilinear spectral multiplier theorem for harmonic oscillator $-\Delta+|x|^2$.
\end{remark}

\subsection{Definition of I-operator}
\begin{definition}[I-operator]
For $N\geqslant1$, and $u\in L^2(M)$. Assume that $n_k$ is the $k$-th eigenvalue of $\sqrt{-\Delta_g}$. We define a smoooth operator $I_N$ by
\begin{align*}
	I_Nu={m_N}(\sqrt{-\Delta_g})u=\sum_{k=1}^{\infty}m(n_k)\pi_ku,
\end{align*}
where
\begin{align*}
	m(\xi)=\begin{cases}
		1, & \mbox{if } \left|\xi\right|\leqslant N, \\
		\big(\frac{N}{\left|\xi\right|}\big)^{1-s}, & \mbox{if} \left|\xi\right|\geqslant2N.
	\end{cases}
\end{align*}
is a smooth function. 
\end{definition}
\begin{remark}\label{Rem-local-wellposed}
It's easy to check that for $b>\frac{1}{2}$,
\begin{align*}
	\Vert u\Vert_{X^{s,b}(\Bbb{R}\times M)}\leqslant C\Vert I_Nu\Vert_{X^{1,b}(\Bbb{R}\times M)}\leqslant N^{1-s}\Vert u\Vert_{X^{s,b}},\hspace{1ex}\forall b>\frac{1}{2},\hspace{1ex}0<s<1.
\end{align*}
and
\begin{align}\label{basic2}
	\Vert u\Vert_{H^s}\leqslant C\Vert I_Nu\Vert_{H^1}\leqslant CN^{1-s}\Vert u\Vert_{H^s}.
\end{align}
For the sake of completeness, we give the proof of $\Vert I_Nu\Vert_{H^1}\leqslant CN^{1-s}\Vert u\Vert_{H^s}$ and others are similar.
\begin{align*}
	\Vert I_Nu\Vert_{H^1(M)}^2&=\sum_{k=1}^{\infty}\langle n_k\rangle^2\Vert I(\pi_ku)\Vert_{L^2}^2\\
	&\sim \sum_{\langle n_k\rangle\leqslant N}\langle n_k\rangle^2\Vert \pi_ku\Vert_{L^2}^2+\sum_{\langle n_k\rangle\geqslant2N}\langle n_k\rangle^2 N^{2(1-s)}\langle n_k\rangle^{2(s-1)}\Vert \pi_ku\Vert_{L^2}^2\\
	&\leqslant C\sum_{\langle n_k\rangle\leqslant N}\langle n_k\rangle^{2-2s}\langle n_k\rangle^{2s}\Vert \pi_ku\Vert_{L^2}^2+\sum_{\langle n_k\rangle\geqslant2N}N^{2-2s}\langle n_k\rangle^{2s}\Vert \pi_ku\Vert_{L^2}^2\\
	&\leqslant N^{2(1-s)}\Vert u\Vert_{H^s}^2.
\end{align*}
 \endproof
\end{remark}



	\section{Bilinear and linear Strichartz estimates}\label{sec:bilstriest}
	In this section, we state  the  bilinear $L^2$ Strichartz estimates and linear Strichartz estimates with mixed $L^p$ norms for solutions to linear Schr\"{o}dinger equation on Zoll manifolds and product spaces $\Bbb{S}^2\times\Bbb{S}^1$, which are frequently used in the proof of main theorem.
\subsection{The Zoll manifold case}
 Before presenting the bilinear Strichartz estimates  on Zoll manifold, we first introduce the properties of eigenvalues associated with Laplace-Beltrami operator on Zoll manifold.
\begin{proposition}[The distribution of eigenvalues of $-\Delta_g$ on Zoll manifold,\cite{Gui}]\label{eigen}
  If $M$ is a Zoll manifold, for which the geodesics of $M$ are $2\pi$ periodic, there exists $\alpha\in\Bbb{N}$ and $E>0$ such that the spectrum of $-\Delta_g$ is contained in the following sets
  \begin{align*}
    \sigma(-\Delta_g)\subset\bigcup_{k=1}^\infty I_k,\quad I_k\stackrel{\triangle}{=} \big[(k+\frac{\alpha}{4})^2-E,(k+\frac{\alpha}{4})^2+E\big].
  \end{align*} 
\end{proposition} 
Now we state the $L_{t,x}^2$ bilinear Strichartz estimates for NLS on three-dimensional Zoll manifolds. This estimate was initially established in the paper of Burq-G\'{e}rard-Tzvetkov\cite{Burq3}.
	\begin{theorem}[Bilinear Strihcartz estimates,\cite{Burq3}]
		Let $M$ be the three-dimensional Zoll manifold. For every compact interval $I\subset \Bbb{R}$, every $\varepsilon>0$, there exists a constant $C>0$ such that for every $N_1\geqslant N_2\geqslant1$ and every $f_1,f_2\in L^2(M)$, the following bilinear estimates hold
		\begin{align}\label{bil-Zoll}
			\big\Vert e^{it\Delta}\Delta_{N_1}f_1e^{it\Delta}\Delta_{N_2}f_2\big\Vert_{L_{t,x}^2(I\times M)}\lesssim N_2^{\frac{1}{2}+\varepsilon}\prod_{j=1}^{2}\Vert f_j\Vert_{L^2(M)}.
		\end{align} 
		Moreover, if $f_1,f_2\in X^{0,b}$ for $b>\frac{1}{2}$, we have
		\begin{align}\label{bil-X}
			\Vert \Delta_{N_1}f_1\Delta_{N_2}f_2\Vert_{L_{t,x}^2(I\times M)}\lesssim N_2^{\frac{1}{2}+\varepsilon}\prod_{j=1}^{2}\Vert\Delta_{N_j} f_j\Vert_{X^{0,b}(M)}.
		\end{align}
	\end{theorem}
	We then use interpolation argument to establish the bilinear estimates  in Bourgain space $X^{0,b(\alpha)}$.
	\begin{lemma}[Interpolation of bilinear estimates]\label{bil-int-ZOll}
	Let $I\subset\Bbb{R}$ be the compact time interval and $N_1\geqslant N_2\geqslant1$.	For  $f_1,f_2\in X^{0,b(\alpha)}$ where $b(\alpha)\in(0,\frac{1}{2})$ satisfying
		\begin{align*}
			f_j=\mathbf{1}_{\sqrt{-\Delta}\in[N_j,2N_j)}f_j,
		\end{align*}  
then we have
		\begin{align}\label{int-bil}
			\Vert f_1f_2\Vert_{L_{t,x}^2(I\times M)}\lesssim N_2^\alpha\prod_{j=1}^{2}\Vert f_j\Vert_{X^{0,b(\alpha)}},
		\end{align}
		where $\alpha\in(\frac{1}{2},\frac{3}{2})$ and $b(\alpha)=\frac{5}{8}-\frac{\alpha}{4}\in(\frac{1}{4},\frac{1}{2})$.
	\end{lemma}
	\proof
	In fact, using H\"{o}lder's inequality, Bernstein's estimate and $\eqref{useful}$, one has
	\begin{align*}
		\Vert f_1f_2\Vert_{L_{t,x}^2}&\lesssim \Vert f_1\Vert_{L_t^4L_x^2(I\times M)}\Vert f_2\Vert_{L_t^4L_x^\infty(I\times M)}\lesssim \Vert f_1\Vert_{X^{0,\frac{1}{4}}}N_2^\frac{3}{2}\Vert f_2\Vert_{L_t^4L_x^2(I\times M)}\\
		&\lesssim N_2^\frac{3}{2}\prod_{j=1}^{2}\Vert f_j\Vert_{X^{0,\frac{1}{4}}}.
	\end{align*}
Then interpolating with bilinear estimates $\eqref{bil-X}$ for $b=\frac{1}{2}+\varepsilon$
\begin{align*}
	\Vert f_1f_2\Vert_{L_{t,x}^2(I\times M)}\lesssim N_2^{\frac{1}{2}+\varepsilon}\prod_{j=1}^{2}\Vert f_j\Vert_{X^{0,b}},
\end{align*} 
we get 
\begin{align*}
	\Vert f_1f_2\Vert_{L_{t,x}^2(I\times M)}\lesssim N_2^\alpha\prod_{j=1}^{2}\Vert f_j\Vert_{X^{0,b(\alpha)}}, 
\end{align*}
where $\alpha\in(\frac{1}{2},\frac{3}{2})$ and $b(\alpha)=\frac{5}{8}-\frac{\alpha}{4}\in(\frac{1}{4},\frac{1}{2})$. 
	\endproof
\begin{corollary}\label{cor-bil-PD-zoll}
Let $u_0,v_0\in L^2(M)$ and $N_1,N_2\in 2^{\Bbb{N}}$ satisfy $N_1\geqslant N_2\geqslant1$. Assume that $P(D)$ and $Q(D)$ are two differential operators on $M$ of orders $m$ and $n$. We have the following bilinear Strichartz estimates
\begin{align}\label{bil-P(D)-zoll}
   \bigg\Vert P(D) e^{it\Delta_g}\Delta_{N_1}u_0Q(D) e^{it\Delta_g}\Delta_{N_2}v_0\bigg\Vert_{L^2([0,1]\times M)}\lesssim N_1^m  N_2^nN_2^{\frac{1}{2}+}\Vert u_0\Vert_{L^2}\Vert v_0\Vert_{L^2},
\end{align} 
Then for $b>\frac{1}{2}$ and $u_0,v_0\in X^{0,b}$, we have
		\begin{align}\label{bil-PD-zoll-X}
 \bigg\Vert P(D) \Delta_{N_1}u_0Q(D)\Delta_{N_2} v_0\bigg\Vert_{L^2([0,1]\times M)}\lesssim N_1^m  N_2^nN_2^{\frac{1}{2}+}\Vert \Delta_{N_1}u_0\Vert_{X^{0,b}}\Vert \Delta_{N_2}v_0\Vert_{X^{0,b}}.	
	\end{align}
\end{corollary}
\begin{proof}
We first give the proof of $\eqref{bil-P(D)-zoll}$, which is the direct consequence of Bernstein's estimate and $\eqref{bil-Zoll}$,
\begin{align*}
& \bigg\Vert P(D)e^{it\Delta_g} \Delta_{N_1}u_0Q(D) e^{it\Delta_g}\Delta_{N_2}v_0\bigg\Vert_{L^2([0,1]\times M)}\\
\lesssim& N_2^{\frac{1}{2}+}\big\Vert P(D)\Delta_{N_1}u_0\big\Vert_{L^2(M)}\big\Vert Q(D)\Delta_{N_2}v_0\big\Vert_{L^2(M)}\\
 \lesssim& N_1^mN_2^nN_2^{\frac{1}{2}+}\Vert \Delta_{N_1}u_0\Vert_{L^2(M)}\Vert\Delta_{N_2}v_0\Vert_{L^2(M)}.
\end{align*} 
Next, we focus on the proof of $\eqref{bil-PD-zoll-X}$. Without loss of generality, it is enough to assume that $f:=\Delta_{N_1}u_0,g:=\Delta_{N_2}v_0\in C_0^\infty([-1,1]\times M)$. Let 
\begin{align*}
  F(t)=e^{-it\Delta_g}f(t),\quad G(t)=e^{-it\Delta_g}g(t).
\end{align*} 
Then we have
\begin{align*}
 & P(D)f(t)=P(D)e^{it\Delta_g}F(t)=\int_{\Bbb{R}}e^{it\tau_1}P(D)e^{it\Delta_g}\widehat{F}(\tau_1)d\tau_1,\\
  &Q(D)g(t)=Q(D)e^{it\Delta_g}G(t)=\int_{\Bbb{R}}e^{it\tau_2}Q(D)e^{it\Delta_g}\widehat{G}(\tau_2)d\tau_2.
  \end{align*}
  By using bilinear estimate $\eqref{bil-P(D)-zoll}$, we can show that 
  \begin{align*}
    \big\Vert P(D)fQ(D)g\big\Vert_{L_{t,x}^2([0,1]\times M)}&\lesssim\bigg\Vert\int_{\Bbb{R}}\int_{\Bbb{R}}e^{it(\tau_1+\tau_2)}P(D)e^{it\Delta_g}\widehat{F}(\tau_1)Q(D)e^{it\Delta_g}\widehat{G}(\tau_2)d\tau_1d\tau_2\bigg\Vert_{L_{t,x}^2([0,1]\times M)}\\
    &\lesssim\int_{\Bbb{R}^2}\bigg\Vert P(D)e^{it\Delta_g}\widehat{F}(\tau_1)Q(D)e^{it\Delta_g}\widehat{G}(\tau_2)\bigg\Vert_{L_{t,x}^2([0,1]\times M)}d\tau_1d\tau_2\\
    &\lesssim N_1^mN_2^nN_2^{\frac{1}{2}+}\int_{\Bbb{R}^2}\big\Vert\widehat{F}(\tau_1)\big\Vert_{L^2}\big\Vert\widehat{G}(\tau_2)\big\Vert_{L^2}d\tau_1d\tau_2\\
    &\lesssim  N_1^mN_2^nN_2^{\frac{1}{2}+}\Vert f\Vert_{X^{0,b}}\Vert g\Vert_{X^{0,b}}.
  \end{align*}
  For the general data $f,g\in X^{0,b}$, one can use the Poisson decomposition to write
  \begin{align*}
    f(t)=\sum_{n\in\Bbb{Z}}\psi(t-\frac{n}{2})f(t)
  \end{align*}
  and it can reduce the estimate to the case $f\in C_0^\infty([-1,1]\times M)$ and hence we finish the proof.
\end{proof}
By interpolating, we can  obtain the following bilinear estimates similar to Lemma \ref{bil-int-ZOll}. 
\begin{lemma}\label{bil-int-PD-Zoll}Let $b(\alpha)=\frac{5}{8}-\frac{\alpha}{4}$ and $\alpha\in(\frac{1}{2},\frac{3}{2})$. For every $u_0,v_0\in X^{0,b(\alpha)}$, we have 
\begin{align}
  \bigg\Vert P(D) \Delta_{N_1}u_0Q(D)\Delta_{N_2} v_0\bigg\Vert_{L^2([0,1]\times M)}\lesssim N_1^m  N_2^nN_2^{\alpha}\Vert \Delta_{N_1}u_0\Vert_{X^{0,b(\alpha)}}\Vert \Delta_{N_2}v_0\Vert_{X^{0,b(\alpha)}}.
\end{align}
\end{lemma}
We also need the  $L^p$ Strichartz estimates on  Zoll manifold, which was proved in \cite{Herr1} by combining the $L^p$ spectral cluster estimate(See Lemma \ref{Sogge-L^p}) and the following exponential sums estimate.
\begin{lemma}[Exponential sums estimate,\cite{Bour-set,Herr1}]\label{expo}
  For $p>4$ and $\alpha$ is a non-negative integer, we have
  \begin{align*}
    \bigg\Vert \sum_{n\in\Bbb{Z}\cap J}c_ne^{-it\mu_n^2}\bigg\Vert_{L^p(I)}\leqslant C(I,\alpha) N^{\frac{1}{2}-\frac{2}{p}}\big(\sum_{n\in\Bbb{Z}\cap J}\left|c_n\right|^2\big)^\frac{1}{2},
  \end{align*}
  where $\mu_n=(n+\frac{\alpha}{4})$ and $J=[b,b+N]$, $b\in\Bbb{R}$, $N\geqslant1$.
\end{lemma}
\begin{lemma}[Sogge's estimate,\cite{Sogge}]\label{Sogge-L^p} Let $M$ be a $d$-dimensional compact manifold without boundary. For all $k\geqslant1$, denote by $\chi_k(\sqrt{-\Delta_g})$ the spectral projector
\begin{align*}
\chi_k(\sqrt{-\Delta_g})=\mathbf{1}_{-\Delta\in[k^2,(k+1)^2]},  
\end{align*}
  then we have
  \begin{align}\label{Sogge-2}
    \big\Vert\chi_k(\sqrt{-\Delta_g})\big\Vert_{L^2(M)\rightarrow L^q(M)}\leqslant Ck^{s(q)},
  \end{align}
  where
  \begin{align*}
    s(q)=\begin{cases}
           \frac{d-1}{2}\big(\frac{1}{2}-\frac{1}{q}\big), & \mbox{if } 2\leqslant q\leqslant\frac{2(d+1)}{d-1}, \\
           \frac{d-1}{2}-\frac{d}{q}, & \mbox{if }\frac{2(d+1)}{d-1}\leqslant q\leqslant\infty.
         \end{cases}
  \end{align*}
\end{lemma}
Next, we   state the $L^p$ Strichartz estimate with the frequency localized initial data. 
\begin{lemma}[$L^p$ Strichartz estimates, \cite{Herr1,Zhao}]\label{L^p-zoll}
  Let $M$ be a  $3$-dimensional Zoll manifold and  $p>4$. Then for all $N\geqslant1$ we have for any subset $I\subset\Bbb{R}$,
  \begin{align*}
    \big\Vert \Delta_Ne^{it\Delta_g}f\big\Vert_{L_{t,x}^p(I\times M)}\lesssim \langle N\rangle^{\frac{3}{2}-\frac{5}{p}}\Vert f\Vert_{L^2(M)}.
  \end{align*}
\end{lemma}

 \subsection{The case of product of sphere $\Bbb{S}^2\times\Bbb{S}^1$}
 In this subsection, we will present the bilinear $L^2$ Strichartz estimates and scale-invariant $L^p$ Strichartz estimates on $\Bbb{S}^2\times\Bbb{S}^1$. 
 
Let us recall the definition of Laplace-Beltrami operator on the product manifold $M=X\times Y$,  
\begin{align*}
  \Delta_M=\Delta_X\otimes I+I\otimes \Delta_Y.
\end{align*}
If $e_1$ and $e_2$ are two eigenfunctions associated with $\Delta_X$ and $\Delta_Y$ respectively, which means
\begin{align*}
  \Delta_X e_1=-\lambda_1^2e_1,\quad \Delta_Ye_2=-\lambda_2^2e_2,
\end{align*}
then the tensor product  $e_1\otimes e_2$ is the eigenfunction of $\Delta_M$ with 
\begin{align*}
  \Delta_M(e_1\otimes e_2)=(\Delta_Xe_1)\otimes e_2+e_1\otimes(\Delta_Ye_2)=-(\lambda_1^2+\lambda_2^2)e_1\otimes e_2.
 \end{align*}
We set $X=\Bbb{S}^2$ and $Y=\Bbb{S}^1$ and  denote $\lambda_{m,n}$ the eigenvalue of $\sqrt{-\Delta_{g}}$ on $M$, then the eigenvalue $\lambda_{m,n}^2$ of Laplacian can be written as
\begin{align*}
  \lambda_{m,n}^2=m^2+n^2+n,\quad m\geq0,\hspace{1ex}n\geq0.
\end{align*}  
Let us denote by $\Pi_n$ the spectral projector on spherical harmonics of degree $n\geq0$ on $\Bbb{S}^2$. For $f(\omega,\theta)\in L^2(\Bbb{S}^2\times \Bbb{S}^1)$, we set
\begin{align*}
  \Theta_mf(\omega)=\frac{1}{2\pi}\int_{0}^{2\pi}f(\omega,\theta)e^{-im\theta}d\theta.
\end{align*}
We can use the spectral resolution to write the linear Schr\"{o}dinger flow as
\begin{align*}
  e^{it\Delta_{g}}f=\sum_{m,n}e^{-it\lambda_{m,n}^2}e^{im\theta}\Pi_{n}\Theta_mf(\omega),\quad(\omega,\theta)\in\Bbb{S}^2\times\Bbb{S}^1.
\end{align*}
We next recall the bilinear Strichartz estimates for solutions to linear Schr\"{o}dinger equation on $M=\Bbb{S}^2\times \Bbb{S}^1$. 
\begin{theorem}[Bilinear Strichartz estimates on $\Bbb{S}^2\times\Bbb{S}^1$, \cite{Burq3}] For every interval $I\subset \Bbb{R}$, every $\varepsilon>0$, there exists $C>0$ such that for every $N_1\geq N_2\geq1$, every $f_1,f_2\in L^2(\Bbb{S}^2\times\Bbb{S}^1)$, we have
  \begin{align}\label{bil-product}
  \big\Vert e^{it\Delta_{g}}\Delta_{N_1}f_1e^{it\Delta_{g}}\Delta_{N_2}f_2\big\Vert_{L_{t,x}^2([0,1]\times\Bbb{S}^2\times\Bbb{S}^1)}\lesssim N_2^{\frac{3}{4}+}\prod_{j=1}^{2}\Vert \Delta_{N_j}f_j\Vert_{L^2(\Bbb{S}^2\times\Bbb{S}^1)}.
  \end{align}
  Moreover, if $f_j\in X^{0,\frac{1}{2}+}([0,1]\times\Bbb{S}^2\times\Bbb{S}^1)$, we have the following bilinear estimate
  \begin{align}\label{bil-bour-pro}
  \big\Vert \prod_{j=1}^{2}\Delta_{N_j}f_j\big\Vert_{L_{t,x}^2([0,1]\times\Bbb{S}^2\times\Bbb{S}^1)}\lesssim N_2^{\frac{3}{4}+}\prod_{j=1}^{2}\Vert\Delta_{N_j}f_j\Vert_{X^{0,\frac{1}{2}+}(\Bbb{S}^2\times\Bbb{S}^1)}.
  \end{align}
\end{theorem}
Interpolating the estimate $\eqref{bil-bour-pro}$ with the following rough estimate
\begin{align*}
  \Vert f_1f_2\Vert_{L_{t,x}^2([0,1]\times\Bbb{S}^2\times\Bbb{S}^1)}\lesssim N_2^\frac{3}{2}\prod_{j=1}^{2}\Vert f_j\Vert_{X^{0,\frac{1}{4}}}
\end{align*}
which holds for $f_j=\mathbf{1}_{\sqrt{-\Delta_{g}}}f_j$,
we have 
\begin{align}\label{bil-interpo-pro}
  \Vert f_1f_2\Vert_{L_{t,x}^2([0,1]\times\Bbb{S}^2\times\Bbb{S}^1)}\lesssim N_2^\beta\prod_{j=1}^{2}\Vert f_j\Vert_{X^{0,b(\beta)}},
\end{align}
where $\beta\in(\frac{3}{4},\frac{3}{2})$ and $b(\beta)=\frac{3}{4}-\frac{\beta}{3}\in(\frac{1}{4},\frac{1}{2})$.

\begin{corollary}
Let $u_0,v_0\in L^2(M)$ and $N_1,N_2\in 2^{\Bbb{N}}$ satisfy $N_1\geqslant N_2\geqslant1$. Assume that $P(D)$ and $Q(D)$ are two elliptic differential operators on $M$ of orders $m$ and $n$. We have the following bilinear Strichartz estimates
\begin{align}\label{bil-P(D)-pro}
   \bigg\Vert P(D) e^{it\Delta_g}\Delta_{N_1}u_0Q(D) e^{it\Delta_g}\Delta_{N_2}v_0\bigg\Vert_{L^2([0,1]\times \Bbb{S}^2\times\Bbb{S}^1)}\lesssim N_1^m  N_2^nN_2^{\frac{3}{4}+}\Vert u_0\Vert_{L^2}\Vert v_0\Vert_{L^2}.
\end{align} 
Then for $b>\frac{1}{2}$ and $u_0,v_0\in X^{0,b}$, we have
		\begin{align}\label{bil-PD-pro-X}
 \bigg\Vert P(D) \Delta_{N_1}u_0Q(D) \Delta_{N_2}v_0\bigg\Vert_{L^2([0,1]\times \Bbb{S}^2\times\Bbb{S}^1)}\lesssim N_1^m  N_2^nN_2^{\frac{3}{4}+}\Vert \Delta_{N_1}u_0\Vert_{X^{0,b}}\Vert \Delta_{N_2}v_0\Vert_{X^{0,b}}.
	\end{align}
\end{corollary}
The proof is similar to Corollary \ref{cor-bil-PD-zoll} and we omit the details. We also have the bilinear estimates in Bourgain space $X^{0,b(\alpha)}$.
\begin{lemma}\label{bil-int-PD-pro}Let $b(\alpha)=\frac{3}{4}-\frac{\alpha}{3}$ and $\alpha\in(\frac{3}{4},\frac{3}{2})$. For every $u_0,v_0\in X^{0,b(\alpha)}$, we have 
\begin{align}
  \bigg\Vert P(D) \Delta_{N_1}u_0Q(D)\Delta_{N_2} v_0\bigg\Vert_{L^2([0,1]\times M)}\lesssim N_1^m  N_2^nN_2^{\alpha}\Vert \Delta_{N_1}u_0\Vert_{X^{0,b(\alpha)}}\Vert \Delta_{N_2}v_0\Vert_{X^{0,b(\alpha)}}.
\end{align}
\end{lemma}

The $L^p$ Strichartz estimates are useful tools in running I-method. To present the $L^p$ Strichartz estimates on product spaces, we first  state the exponential sums estimate which will be used in the proof of $L^p$ Strichartz estimates.
\begin{lemma}[Exponential sums,\cite{HS15}]\label{exponential}Let $p>4$. Then for all $N\geqslant1$, $\tau_0\subset\Bbb{R}$, $a\in\ell^2(\Bbb{Z}^2)$, $z\in\Bbb{Z}^2$ and $\mathcal{S}_N\subset z+\{k\times k\in\Z^2:-N\leqslant k\leqslant N\}$, the following estimate holds true,
  \begin{align}\label{expo-pro}
    \bigg\Vert\sum_{(m,n)\in\mathcal{S}_N}e^{-it\lambda_{m,n}e^{im\theta}}a_{m,n}\bigg\Vert_{L_{t,\theta}^p(\tau_0\times\Bbb{S})}\lesssim N^{1-\frac{3}{p}}\Vert a\Vert_{\ell^2(\Bbb{Z}^2)}.
  \end{align} 
\end{lemma}
We now give the $L^p$ Strichartz estimates on product of spheres.
\begin{lemma}\label{L^p-pro}Let $p>4$. For all $N\geqslant1$, we have the following $L^p$ Strichartz estimate 
\begin{align*}
  \big\Vert e^{it\Delta_{g}}\Delta_{N}f\big\Vert_{L_{t,x}^p([0,1]\times\Bbb{S}^2\times\Bbb{S}^1)}\leqslant N^{s_1(p)}\Vert f\Vert_{L^2(\Bbb{S}^2\times\Bbb{S}^1)},
\end{align*}
where
\begin{align*}
  s_1(p)=\begin{cases}
           \frac{3}{2}-\frac{5}{p}, & \mbox{if } p\geqslant6, \\
           \frac{5}{4}-\frac{7}{2p}, & \mbox{if }4<p\leqslant 6.
         \end{cases}
\end{align*}
\end{lemma}
\begin{proof}
  First using the spectral representation, we can expand $e^{it\Delta_{\Bbb{S}^2\times\Bbb{S}^1 }}$ to
  \begin{align*}
    e^{it\Delta_{g}}f=\sum_{m,n}e^{-it\lambda_{m,n}^2}e^{im\theta}\Pi_{n}\Theta_mf(\omega),\quad(\omega,\theta)\in\Bbb{S}^2\times\Bbb{S}^1.
  \end{align*}
  We denote $a_{m,n}=\Pi_{n}\Theta_mf(\omega)$ and using Lemmas \ref{exponential} and \ref{Sogge-L^p}, we have
  \begin{align*}  
  \big\Vert e^{it\Delta_{\Bbb{S}^2\times\Bbb{S}^1}}\Delta_{N}f\big\Vert_{L_{t,x}^p([0,1]\times\Bbb{S}^2\times\Bbb{S}^1)}&\lesssim\bigg\Vert\big\Vert\sum_{m,n}e^{-it\lambda_{m,n}^2}e^{im\theta}\Pi_{n}\Theta_mf(\omega)\big\Vert_{L_t^p([0,1],L_\theta^p(\Bbb{S}^1))}\bigg\Vert_{L_\omega^p(\Bbb{S}^2)}\\
   &\lesssim N^{1-\frac{3}{p}}\bigg(\sum_{N\leqslant\langle\lambda_{m,n}\rangle<2N}\big\Vert\Pi_{n}\Theta_mf(\omega)\big\Vert_{L^p(\Bbb{S}^2)}^2\bigg)^\frac{1}{2}\\ 
  &\lesssim N^{1-\frac{3}{p}}\bigg(\sum_{N\leqslant\langle\lambda_{m,n}\rangle<2N}\langle\lambda_{m,n}\rangle^{2s(p) }\big\Vert\Pi_{n}\Theta_mf(\omega)\big\Vert_{L^2(\Bbb{S}^2)}^2\bigg)^\frac{1}{2}\\
  &\lesssim N^{s(p)+1-\frac{3}{p}}\Vert f\Vert_{L^2(\Bbb{S}^2\times\Bbb{S}^1)}.
  \end{align*}
  We set $s_1(p)=s(p)+1-\frac{3}{p}$, where $s(p)$ is as in Lemma \ref{Sogge-L^p}  and hence complete the proof.
\end{proof}


\section{The proof of Theorem \ref{Thm1}}\label{sec:prothm}
In this section, we use I-method to prove Theorem $\ref{Thm1}$. For simiplicity, we denote $Iu$ by $I_Nu$.  To start with, we will prove the local well-posedness of the modified equation, where  $v$ denotes  $Iu$, which satisfies
\begin{align}\label{I}
\begin{cases}
	i\partial_tv+\Delta_g v=I(\left|u\right|^2u)=I(\left|I^{-1}v\right|^2I^{-1}v),&(t,x)\in\Bbb{R}\times M,\\
	v(0)=v_0(x)=Iu_0, 
\end{cases}
\end{align}
where $M$ is Zoll manifold or $\Bbb{S}^2\times\Bbb{S}^1$. 
\subsection{The local well-posedness of the I-system}

We notice that $\Vert v_0\Vert_{H^1(M)}\leqslant N^{1-s}\Vert u_0\Vert_{H^s}$. In order to prove the local well-posedness of $Iu$, we need to control the term $I(\left|u\right|^2u)$, which is given by the following lemma.
\begin{lemma}[Nonlinear estimates]\label{I-estimate}
\begin{enumerate}
  \item When $M$ is Zoll manifold. Let $\frac{1}{2}<s\leqslant1$ and $T>0$. There exists $C>0$, $T\in[0,1)$ and  $b(s-)\in(\frac{3}{8},\frac{1}{2})$ such that for  $u\in X_T^{s,b(s-)}$, the cubic nonlinear estimate holds
\begin{align}
	\Vert I(\left|u\right|^2u)\Vert_{X_T^{1,-b(s-)}}\leqslant C\Vert Iu\Vert^3_{X_T^{1,b(s-)}},
\end{align}
where $b(s-)=\frac{5}{8}-\frac{s}{4}+$.
\item When $M=\Bbb{S}^2\times\Bbb{S}^1$. Let $\frac{3}{4}<s\leqslant1$ and $T>0$. There exists $C>0$, $T\in[0,1)$ and  $b'=b(s-)\in(\frac{5}{12},\frac{1}{2})$ such that for  $u\in X_T^{s,b(s-)}$, the cubic nonlinear estimate holds
\begin{align}
	\Vert I(\left|u\right|^2u)\Vert_{X_T^{1,-b(s-)}}\leqslant C\Vert Iu\Vert^3_{X_T^{1,b(s-)}}.
\end{align}
\end{enumerate}
\end{lemma}
We will prove Lemma \ref{I-estimate} in Subsection \ref{nonlinear}. As a direct consequence, we can obtain the local well-posedness for the Cauchy problem $\eqref{I}$.
\begin{proposition}[Local well-posedness]\label{Local-M}
\begin{enumerate}
\item When $M$ is Zoll manifold. Let $s\in(\frac{1}{2},1)$ and $v_0=Iu_0\in H^1(M)$. Then, there exists $\delta=C( \Vert Iu_0\Vert_{H^1(M)})$ such that  the Cauchy problem $\eqref{I}$ is local well-posed on the  small time interval $[0,\delta]$ with $\delta\sim N^{-\frac{4(1-s)}{2s-1}-}$, and
\begin{equation}
\|Iu\|_{X^{1,\frac12+}([0,\delta]\times M)}\lesssim N^{1-s}.
\end{equation}
    \item When $M=\Bbb{S}^2\times\Bbb{S}^1$. Let $s\in(\frac{3}{4},1)$ and $v_0=Iu_0\in H^1(M)$. Then, there exists $\delta=C( \Vert Iu_0\Vert_{H^1(M)})$ such that  the Cauchy problem $\eqref{I}$ is local well-posed on the  small time interval $[0,\delta]$ with $\delta\sim N^{-\frac{6(1-s)}{4s-3}-}$, and
\begin{equation}
\|Iu\|_{X^{1,\frac12+}([0,\delta]\times M)}\lesssim N^{1-s}.
\end{equation}
\end{enumerate}
\end{proposition}
\begin{proof}
$Iu$ satisfies the following integral equation on time interval $[0,\delta]$ via the Duhamel formula,
\begin{align*}
	v(t)=Iu(t)=e^{it\Delta}Iu_0-i\int_{0}^{\delta}e^{i(t-s)\Delta}I(\left|u\right|^2u)ds.
\end{align*}
When $M$ is Zoll manifold, using Remark \ref{Rem-local-wellposed}, we have the following contraction argument
\begin{align*}
\Vert v\Vert_{X^{1,\frac{1}{2}+}}&\lesssim \Vert v_0\Vert_{H^1(M)}+C\delta^{1-b-b(s-)}\Vert I(\left|I^{-1}v\right|^2I^{-1}v)\Vert_{X^{1,-b(s-)}}\\
&\lesssim \Vert Iu_0\Vert_{H^1(M)}+C\delta^{1-b-b(s-)}\Vert v\Vert_{X^{1,b(s-)}}^3\\
&\lesssim \Vert Iu_0\Vert_{H^1(M)}+C\delta^{1-b-b(s-)}\delta^{3(b-b(s-)-\varepsilon)}\Vert v\Vert_{X^{1,\frac{1}{2}+}}^3\\
&\lesssim N^{1-s}+C\delta^{s-\frac{1}{2}-}\Vert v\Vert_{X^{1,\frac{1}{2}+}}.
\end{align*}
Then using the continuity method, we choose $\delta\in(0,1)$ and $b(s-)=\frac{5}{8}-\frac{s-}{4}$, then 
\begin{align*}
	\delta^{s-\frac{1}{2}-}\sim \Vert Iu_0\Vert_{H^1(M)}^{-2}\Longleftrightarrow \delta\sim \Vert Iu_0\Vert_{H^1(M)}^{-\frac{4}{2s-1}-}\gtrsim N^{-\frac{4(1-s)}{2s-1}-}.
\end{align*}
The similar calculus give the bounds for $\delta$ when $M=\Bbb{S}^2\times\Bbb{S}^1$:
\begin{align*}
  \delta\sim\Vert Iu_0\Vert_{H^1}^{-\frac{6}{4s-3}-}\sim N^{-\frac{6(1-s)}{4s-3}-},
\end{align*}
where we choose $b(s)=\frac{3}{4}-\frac{s-}{3}$. 
\end{proof}
\subsection{Nonlinear estimates in the I-system}\label{nonlinear}
In this part, we give the proof of Lemma \ref{I-estimate} by the bilinear estimates and almost orthogonal property $\eqref{orthogonal}$. By duality argument, it is sufficient to prove 
\begin{align}\label{non-dual}
\left|\int_{\Bbb{R}\times M}\overline{\phi}I(\left|u\right|^2u)dxdt\right|\lesssim \Vert \phi\Vert_{X^{-1,b(s-)}}\Vert Iu\Vert_{X^{1,b(s-)}}^3
\end{align}
for $\phi\in X^{-1,b'}(\Bbb{R}\times M)$. Using the similar localization argument in \cite{Burq2}, we denote 
\begin{gather*}
\phi^{N_0}=\sum_{N_0\leqslant\langle n_k\rangle<2N_0}\pi_k\phi,\\
u^{N_j}=\sum_{N_j\leqslant\langle n_k\rangle<2N_j}\pi_ku,
\end{gather*}
where $\pi_k$ denote the spectral projector. Then we can denote the localized version of the left-hand side of $\eqref{non-dual}$ as
\begin{align*}
	J(N_0,N_1,N_2,N_3)=\left|\int_{\Bbb{R}\times M}\overline{\phi}^{N_0}m(N_0)u^{N_1}\overline{u}^{N_2}u^{N_3}dxdt\right|
\end{align*}
and $\eqref{non-dual}$ can be rewritten as the sum of $J(N_0,N_1,N_2,N_3)$
\begin{align*}
	&\left|\int_{\Bbb{R}\times M}\overline{\phi}I(\left|u\right|^2u)dxdt\right|\\
	\lesssim&\sum_{N_1,N_2,N_3,N_0\in2^\Bbb{Z}}J(N_0,N_1,N_2,N_3)\\
	\lesssim&\sum_{N_0\geqslant C(N_1+N_2+N_3)}J(N_0,N_1,N_2,N_3)+\sum_{N_0\leqslant C(N_1+N_2+N_3)}J(N_0,N_1,N_2,N_3)\\
	\triangleq&J_1(N_0,N_1,N_2,N_3)+J_2(N_0,N_1,N_2,N_3).
\end{align*}
Using the symmetry we can assume that $N_1\geqslant N_2\geqslant N_3\geqslant1$.  We  use $\eqref{orthogonal}$ and the similar argument as Lemma 2.8 in \cite{Burq2} to show 
\begin{align*}
\sum_{N_0>C(N_1+N_2+N_3)}	J_1(N_0,N_1,N_2,N_3)\lesssim  \Vert Iu\Vert_{X^{1,b'}}^3\Vert \phi\Vert_{X^{-1,b(s-)}}.
\end{align*}

Now we focus on the regime $N_0\leqslant C(N_1+N_2+N_3)$, i.e. $J_2(N_0,N_1,N_2,N_3)$. Under the assumption of $N_j$, we actually have $N_0\lesssim N_1$. Using the bilinear estimates $\eqref{int-bil}$ and H\"{o}lder's inequality, we have
\begin{align*}
	J_2(N_0,N_1,N_2,N_3)&\lesssim\frac{m(N_0)}{m(N_1)m(N_2)m(N_3)}\big\Vert \phi^{N_0}Iu^{N_2}\big\Vert_{L_{t,x}^2(\Bbb{R}\times M)}\big\Vert \phi^{N_1}Iu^{N_3}\big\Vert_{L_{t,x}^2(\Bbb{R}\times M)}\\
	&\lesssim\frac{m(N_0)}{m(N_1)m(N_2)m(N_3)}(N_2N_3)^\alpha\Vert \phi^{N_0}\Vert_{X^{0,b(\alpha)}}\prod_{j=0}^{3}\Vert Iu^{N_j}\Vert_{X^{0,b(\alpha)}}.
\end{align*} 
We choose $\alpha=s-$ and $s\in(\frac{1}{2},1)$ when $M$ is Zoll manifold and $\alpha=s-$, $s\in(\frac{3}{4},1)$ when $M=\Bbb{S}^2\times\Bbb{S}^1$. Then we have $b(s-)\in(\frac{3}{8},\frac{1}{2})$ when $M$ is Zoll manifold and $b(s-)\in(\frac{5}{12},\frac{1}{2})$ when $M=\Bbb{S}^2\times\Bbb{S}^1$.
Then we get
\begin{align*}
	J_2(N_0,N_1,N_2,N_3)&\lesssim\frac{m(N_0)N_0}{m(N_1)N_1}\frac{(N_2N_3)^{s-1-\varepsilon}}{m(N_2)m(N_3)}\Vert \phi^{N_0}\Vert_{X^{-1,b(s-)}}\prod_{j=1}^{3}\Vert I_Nu^{N_j}\Vert_{X^{1,b(s-)}}.
\end{align*}
Using the fact that 
\begin{align*}
	\frac{m(N_0)N_0}{m(N_1)N_1}&\lesssim\begin{cases}
		\frac{N_0}{N_1},&\mbox{if }N_0\lesssim N_1\ll N\\
		\frac{N_0}{N_1^sN^{1-s}},&\mbox{if }N_0\ll N\lesssim N_1\\
		\frac{N_0^s}{N_1^s},&\mbox{if }N\lesssim N_0\lesssim N_1
	\end{cases}\\
	&\lesssim\big(\frac{N_0}{N_1}\big)^s
\end{align*}
and 
\begin{align*}
	\frac{N_i^{s-1-\varepsilon}}{m(N_i)}=\begin{cases}
		N_i^{s-1-\varepsilon},&\mbox{if }N_i\ll N,\\
		N^{s-1}N_i^{0-},&\mbox{if }N_i\gtrsim N,
	\end{cases}
\end{align*}
we have 
\begin{align*}
	&J_2(N_0,N_1,N_2,N_3)\\
	\lesssim&\sum_{N_0\lesssim N_1+N_2+N_3}\frac{m(N_0)N_0}{m(N_1)N_1}\frac{(N_2N_3)^{s-1-}}{m(N_2)m(N_3)}\Vert \phi^{N_0}\Vert_{X^{-1,b(s-)}}\prod_{j=1}^{3}\Vert Iu^{N_j}\Vert_{X^{1,b(s-)}}\\
	\lesssim&\sum_{N_0\lesssim N_1}\frac{m(N_0)N_0}{m(N_1)N_1}\Vert \phi^{N_0}\Vert_{X^{-1,b(s-)}}\Vert Iu^{N_1}\Vert_{X^{1,b(s-)}}\Vert I_Nu\Vert_{X^{1,b(s-)}}^2\\
	\lesssim &\Vert \phi\Vert_{X^{-1,b(s-)}}\Vert Iu\Vert_{X^{1,b(s-)}}^3.
	\end{align*}
	Hence, we complete the proof of Lemma \ref{I-estimate}.
	\endproof
\subsection{Almost conservation law and energy increment}
In this subsection, we will prove the almost conservation law of energy. 
\begin{proposition}[Energy increment]\label{prop}
Let $s\in(0,1)$, $u_0\in H^s(M)$. Let $u(t,x)$ be solution to the Cauchy problem $\eqref{I}$ and $\Vert Iu\Vert_{X^{1,\frac{1}{2}+}([0,\delta]\times M)}\lesssim N^{1-s}$, then we have the increment of modified energy $E(Iu)$,
\begin{align*}
&\hspace{5ex}	\left|E(Iu(t))-E(Iu(0))\right|\\
&\lesssim \begin{cases}\big(N^{-1+}+\delta^{\frac{1}{4}-}N^{-\frac{1}{2}+}\big)\Vert Iu\Vert_{X^{1,b}}^4+N^{-\frac{3}{2}+}\Vert Iu\Vert_{X^{1,b}}^6,\quad \text{if}\; M= \text{Zoll manifold,}&s>\frac{\sqrt{21}-1}{4},\\
\big(N^{-\frac{1}{2}+}+\delta^{\frac{1}{6}-}N^{-\frac{1}{4}+}\big)\Vert Iu\Vert_{X^{1,b}}^4+N^{-\frac{5}{4}+}\Vert Iu\Vert_{X^{1,b}}^6\quad \text{if}\; M=\Bbb{S}^2\times \Bbb{S}^1, &s>\frac{1+3\sqrt{5}}{8},\end{cases}
\end{align*}
for all $t\in[0,\delta]$.
\end{proposition}
For the Caychy prolem $\eqref{I}$, the energy $E(I(u))(t)$ is not conserved any more, we need to calculate the increment on time intervals $[0,\delta]$.

We first introduce the modified energy $E(Iu)(t)$,
\begin{align*}
E(Iu)(t)=\frac{1}{2}\int_{M}\left|\nabla_g Iu\right|^2dx+\frac{1}{4}\int_{M}\left|Iu\right|^4dx.
\end{align*}
Before the calculation of the change of the modified energy, we show that the energy $E(u)$ is conserved.
\begin{align*}
\partial_t E(u)(t)=\Re\int_{M}\partial_t\overline{u}((-\Delta)
u+\left|u\right|^2u)dx=\Re \int_{M}\overline{u_t}(\left|u\right|^2u+(-\Delta)
u-iu_t)dx=0.
\end{align*}
We differentiate in time to obtain
\begin{align*}
\frac{d}{dt}E(Iu)(t)&=\Re\int_{M}\partial_t\overline{Iu}((-\Delta)Iu-iIu_t+\left|Iu\right|^2Iu)dx\\
&=\Re\int_{M}\partial_t\overline{Iu}(-I(\left|u\right|^2u)+\left|Iu\right|^2Iu)dx,
\end{align*}
where the last step we use the structure of Cauchy problem $\eqref{I}$. Using the Fundamental Theorem of Calculus and writing $u$ as $u=\sum_{n}\pi_nu(t,x)$, we obtain
\begin{align*}
E(Iu(t))-E(Iu(0))&=\Re\sum_{n_i}\int_{0}^{t}\int_{M}[m(n_i)\overline{\pi_{n_1}u_t}m(n_2)m(n_3)m(n_4)\pi_{n_2}u\overline{\pi_{n_3}u}\pi_{n_4}u\\
&\hspace{2ex}-m(n_1)\overline{\pi_{n_1}u_t}I(\pi_{n_2}u\overline{\pi_{n_3}u}\pi_{n_4}u)] dxdt.
\end{align*}
Since $I$ is  a self-adjoint operator, we can transfer it to   $\overline{\pi_{n_1}u_t}$. Then we can write it as
\begin{align*}
&\hspace{5ex}E(Iu(t))-E(Iu(0))\\
&=\Re\sum_{n_i}\int_{0}^{t}\int_{M}\left[1-\frac{m(n_1)}{m(n_2)m(n_3)m(n_4)}\right]\pi_{n_1}(Iu_t)\pi_{n_2}(I_Nu)\overline{\pi_{n_3}(I_Nu)}\pi_{n_4}(I_Nu)dxdt\\
&=\Re i\sum_{n_i}\int_{0}^{t}\int_{M}\left[1-\frac{m(n_1)}{m(n_2)m(n_3)m(n_4)}\right]\overline{\pi_{n_1}(\Delta
	Iu)}\pi_{n_2}(Iu)\overline{\pi_{n_3}(Iu)}\pi_{n_4}(Iu)dxdt\\
&\hspace{2ex}-\Re i\int_{0}^{t}\int_{M}\left[1-\frac{m(n_1)}{m(n_2)m(n_3)m(n_4)}\right]\overline{\pi_{n_1}(I(\left|u\right|^2u))}\pi_{n_2}(Iu)\overline{\pi_{n_3}(Iu)}\pi_{n_4}(Iu)dxdt\\
&\stackrel{\triangle}{=}\Re(iI_1-iI_2),
\end{align*}
where
\begin{gather*}
I_1=\sum_{n_i}\int_{0}^{t}\int_{M}\left[1-\frac{m(n_1)}{m(n_2)m(n_3)m(n_4)}\right]\overline{\pi_{n_1}(\Delta Iu)}\pi_{n_2}(Iu)\overline{\pi_{n_3}(Iu)}\pi_{n_4}(Iu)dxdt,\\
I_2=\sum_{n_i}\int_{0}^{t}\int_{M}\left[1-\frac{m(n_1)}{m(n_2)m(n_3)m(n_4)}\right]\overline{\pi_{n_1}(I(\left|u\right|^2u))}\pi_{n_2}(Iu)\overline{\pi_{n_3}(Iu)}\pi_{n_4}(Iu)dxdt.
\end{gather*}
\subsubsection{Bound on $I_1$}\label{subsubs-bound-I1}  We start by giving the bound on $I_1$. We split $u$ into Littlewood-Paley pieces $u=\sum\limits_{N}u_N$ where $N\in2^{\Bbb{N}}$ and $u_N=P_{[N,2N)}u$ for $N>1$ and we add the definition of $u_1=P_{[1,2)}u$. By the standard localization argument, we can easily reduce the estimate to an integral involving dyadic pieces at scale $N_j$, $j=1,2,3,4$. If we obtain the localized estimate which contains the geometric decay for highest frequency, then we can sum over all frequencies $N_j$. First, we notice that
\begin{align*}
\Vert \Delta Iu\Vert_{X^{-1,\frac{1}{2}+}}\lesssim\Vert Iu\Vert_{X^{1,\frac{1}{2}+}}.
\end{align*}
Thus, we can reduce $I_1$ to that
\begin{align}\label{Goal1}
& \left|\sum_{n_i\sim N_i}\int_{0}^{t}\int_{M}\left[1-\frac{m(n_1)}{m(n_2)m(n_3)m_(n_4)}\right]\overline{\pi_{n_1}\phi_1}\pi_{n_2}\phi_2\overline{\pi_{n_3}\phi_3}\pi_{n_4}\phi_4dxdt\right|\notag\\
&\hspace{22ex}\lesssim C(\delta,N)(N_1N_2N_3N_4)^{0-}\Vert\phi_1\Vert_{X^{-1,\frac{1}{2}+}}\prod_{j=2}^{4}\Vert\phi_i\Vert_{X^{1,\frac{1}{2}+}},
\end{align}
where
\begin{align*}
  C(\delta,N)=\begin{cases}
    N^{-1+}+N^{-\frac{1}{2}+}\delta^{\frac{1}{4}-}, & \mbox{if } \mbox{ $M$ is Zoll manifold}, \\
    N^{-\frac{1}{2}+}+N^{-\frac{1}{4}+}\delta^{\frac{1}{6}-}, & \mbox{if }\hspace{1ex} M=\Bbb{S}^2\times\Bbb{S}^1
  \end{cases}
\end{align*}
 and $\phi_j$ are spectrally localized on $n_i\sim N_i$. We will not distinguish $\phi$ and its complex conjugate since the  existence of the complex conjugates  will not affect the analysis above. Without loss of generality, we assume that $N_2\geqslant N_3\geqslant N_4$.

We now use the case by case analysis which relies on the order relation of  $N_2$, $N_1$ and $N$. We consider the several cases as follows,
\begin{enumerate}
\item \textbf{Case One(1)}: $N_1\geqslant CN_2$, $N_2\geqslant N_3\geqslant N_4$,
\item \textbf{Case One(2)}: $N_2\sim \max\{N_1,N_3\}$,
\item \textbf{Case Two(1)}: $N_1\lesssim N_2$, $N_2\ll N$,
\item    \textbf{Case Two(2)}: $N_1\lesssim N_2$, $N_2\geqslant N\gg N_3\gg N_4$,
\item    \textbf{Case Three  }: $N_1\lesssim N_2$, $N_2\sim N_3\gtrsim N$,
\begin{align*}
	\begin{cases}
		N_2\sim N_3, N_1\gg N,  \\
		N_2\sim N_3\gtrsim N,\quad N_1\lesssim N.
	\end{cases}
\end{align*}
\item \textbf{Case Four}: $N_1\lesssim N_2$, $N_2\gg N_3\gtrsim N$.
\end{enumerate}
We begin to consider Case One(1)(2), which are two similar cases. In these cases, we will use the almost orthogonal property of eigenfunctions(Lemma \ref{orthogonal}).
\begin{lemma}\label{CaseOne(1)}
Let $N_1\geqslant CN_2$. There exists $C>0$ such that for every $q>0$, we have
\begin{align*}
	\text{(LHS) of }\eqref{Goal1}\lesssim \frac{1}{N_1^q}\prod_{j=2}^{4}\Vert \phi_j\Vert_{X^{1,\frac{1}{2}+}}.
\end{align*}
\end{lemma}
\begin{proof}
This lemma is the direct consequence of Lemma \ref{orthogonal}. Applying the almost orthogonal estimate, we get
\begin{align*}
	&\hspace{5ex} \mbox{(LHS) of \eqref{Goal1}}\\
	&=\sum_{n_i\sim N_i}\left|\bigg(1-\frac{m(n_1)}{m(n_2)m(n_3)m(n_4)}\bigg)\int_{0}^{t}\int_{M}\overline{\pi_{ n_1}{\phi}_1}\pi_{ n_2}{\phi}_2\overline{\pi_{ n_3}{\phi}_3}\pi_{ n_4}{\phi}_4dxdt\right|\\
	&\lesssim\sum_{n_i\sim N_i}\frac{1}{( n_1)^q}\left|1-\frac{m(n_1)}{m(n_2)m(n_3)m(n_4)}\right|\int_{0}^{t}\prod_{j=1}^{4}\big\Vert\pi_{ n_j}{\phi}_j\big\Vert_{L^2(M)}dt.
\end{align*}
Since $m(n_i)$ is a non-increasing function and $n_1\geqslant Cn_2$, we have $\frac{m(n_1)}{m(n_2)}\lesssim1$. We notice that $m(n_i)/n_i\leqslant C$, thus we have the rough bound
\begin{align*}
	\left|1-\frac{m(n_1)}{m(n_2)m(n_3)m(n_4)}\right|\lesssim1+\frac{m(n_1)}{m(n_2)m(n_3)m(n_4)}\lesssim n_1^2,
\end{align*}
where the direct calculus gives
\begin{align*}
	\frac{m(n_1)}{m(n_2)m(n_3)m(n_4)}=\begin{cases}
		\big(\frac{N}{n_1}\big)^{1-s}, & \mbox{if } n_4\leqslant n_3\leqslant n_2\leqslant N\lesssim n_1, \\
		\big(\frac{n_2}{n_1}\big)^{1-s}, & \mbox{if }n_4\leqslant n_3\leqslant N\lesssim n_2\lesssim n_1,  \\
		\big(\frac{n_2n_3}{Nn_1}\big)^{1-s}, & \mbox{if } n_4\leqslant N\leqslant n_3\leqslant
		n_2\lesssim n_1 \\
		\big(\frac{n_2n_3n_4}{n_1N^2}\big)^{1-s}, & \mbox{if }N\leqslant n_4\leqslant n_3\leqslant n_2\lesssim n_1.
	\end{cases}
\end{align*}
Recall the Weyl's law\cite[Chapter 4]{Sogge} on $M$,
\begin{align*}
	\sharp\{\nu:N_j\leqslant\sqrt{\nu}\leqslant2N_j,\quad \nu\mbox{ is the eigenvalue of }-\Delta_g\}\lesssim N_j^3.
\end{align*}
Applying the Cauchy-Schwarz inequality, we have
\begin{align*}
	\sum_{n_i\sim N_i}\Vert\pi_{ N_i}{\phi}_i\Vert_{L^2(M)}\lesssim N_i^\frac{3}{2}\Vert{\phi}_i\Vert_{L^2(M)}.
\end{align*}
Using H\"{o}lder's inequality, we obtain
\begin{align*}
	(LHS)\mbox{ of \ref{Goal1}}&\lesssim\frac{1}{ N_1^{q-2}}\int_{0}^{t}\prod_{j=1}^{4} N_j^\frac{3}{2}\Vert{\phi}_j\Vert_{L^2(M)}dt\lesssim\frac{1}{N_1^{q-8}}\int_{0}^{t}\prod_{j=1}^{4}\Vert\phi_j\Vert_{L^2(M)}dt\\
	&\lesssim\frac{1}{N_1^{q-8}}\Vert\phi_1\Vert_{L_{t,x}^2([0,t]\times M)}\Vert\phi_2\Vert_{L_{t,x}^2([0,t]\times M)}\Vert\phi_3\Vert_{L_t^\infty([0,t];L^2(M))}\\
	&\hspace{2ex}\times\Vert \phi_4\Vert_{L_t^\infty([0,t];L^2(M))}\\
	&\lesssim\frac{1}{N_1^{q-8}}\Vert \phi_1\Vert_{X^{0,\frac{1}{2}+}}\prod_{j=2}^{4}\Vert \phi_j\Vert_{X^{0,\frac{1}{2}+}}\\
	&\lesssim\frac{1}{N_1^{q-10-}}\Vert \phi_1\Vert_{X^{-1,\frac{1}{2}+}}\prod_{j=2}^{4}\Vert\phi_j\Vert_{X^{1,\frac{1}{2}+}}. 
\end{align*}
We can choose $q\gg1$ large enough which complete the proof.
\end{proof}
\begin{remark}
 Using the similar argument, we can prove  Case One(2) with the roles of $N_1$ and $N_2$ interchanged.
\end{remark}
Now, we can assume that $N_1\lesssim N_2$ in the rest of Subsection \ref{subsubs-bound-I1}.\\
\noindent\textbf{Case Two(1)}: $N_2\ll N$. In this subcase, all the frequencies cannot comparable to $N$, which is the trivial case,
\begin{align*}
1-\frac{m(n_1)}{m(n_2)m(n_3)m(n_4)}=0.
\end{align*}
\noindent\textbf{Case Two(2)}: $N_1\sim N_2\gtrsim N\gg N_3\geqslant N_4$. If $N_1$ or $N_2$ is the dominating term, then it can be reduced to the Case One. Thus we can assume that $N_1\sim N_2$ in this case. We split the intervals $[N_1,2N_1)$ into several sub-intervals $I_\alpha$ where the length of each sub-interval is $N_3$ and the number of sub-intervals is $J$. We then split $[N_2,2N_2)$ into $K$ sub-intervals $I_\beta$ with $N_3$ the length of each intervals. Noticing that $J\sim K\sim\frac{N_2}{N_3}$. Thus we can rewrite the (LHS) of $\eqref{Goal1}$ as
\begin{align}\label{split1}
\left|\sum_{I_\alpha,I_\beta}\sum_{\substack{n_1\in I_\alpha,n_2\in I_\beta\\n_3\sim N_3,n_4\sim N_4}}\int_{0}^{t}\int_{M}\left[1-\frac{m(n_)}{m(n_2)m(n_3)m(n_4)}\right]\overline{\pi_{n_1}(\phi_1)}\pi_{n_2}\phi_2\overline{\pi_{n_3}\phi_3}\pi_{n_4}\phi_4dxdt\right|.
\end{align}
To estimate the above term, we need to split the sum of $I_\alpha$ and $I_\beta$ to several cases:
\begin{enumerate}
\item \textbf{Sub-Case 1}\hspace{1ex}($S_1$): the elements in $I_\alpha$ are at least $CN_3$ larger than those in $I_\beta$, that is $\forall  n_1\in I_\alpha,n_2\in I_\beta$, we have $n_1-n_2\geqslant CN_3$,
\item \textbf{Sub-Case 2}\hspace{1ex}($S_2$): the elements in $I_\beta$ are at least $CN_3$ larger than those in $I_\alpha$, that is $\forall n_1\in I_\alpha,n_2\in I_\beta$, we have $n_2-n_1\geqslant CN_3$,
\item \textbf{Sub-Case 3}\hspace{1ex}($S_3$): the distance of two intervals is smaller than $CN_3$, that is $\forall n_1\in I_\alpha,n_2\in I_\beta$, we have $\left| n_1-n_2\right|< CN_3$.
\end{enumerate}
Then we can write $\eqref{split1}$ as the sum of three terms,
\begin{align*}
\eqref{split1}&\leqslant\left|\sum_{S_1}\sum_{\substack{n_1\in I_\alpha,n_2\in I_\beta\\n_3\sim N_3,n_4\sim N_4}}\int_{0}^{t}\int_{M}\left[1-\frac{m(n_)}{m(n_2)m(n_3)m(n_4)}\right]\overline{\pi_{n_1}(\phi_1)}\pi_{n_2}\phi_2\overline{\pi_{n_3}\phi_3}\pi_{n_4}\phi_4dxdt\right|\\
&\hspace{2ex}+\left|\sum_{S_2}\sum_{\substack{n_1\in I_\alpha,n_2\in I_\beta\\n_3\sim N_3,n_4\sim N_4}}\int_{0}^{t}\int_{M}\left[1-\frac{m(n_)}{m(n_2)m(n_3)m(n_4)}\right]\overline{\pi_{n_1}(\phi_1)}\pi_{n_2}\phi_2\overline{\pi_{n_3}\phi_3}\pi_{n_4}\phi_4dxdt\right|\\
&\hspace{2ex}+\left|\sum_{S_3}\sum_{\substack{n_1\in I_\alpha,n_2\in I_\beta\\n_3\sim N_3,n_4\sim N_4}}\int_{0}^{t}\int_{M}\left[1-\frac{m(n_)}{m(n_2)m(n_3)m(n_4)}\right]\overline{\pi_{n_1}(\phi_1)}\pi_{n_2}\phi_2\overline{\pi_{n_3}\phi_3}\pi_{n_4}\phi_4dxdt\right|\\
&\stackrel{\triangle}{=}J_1+J_2+J_3.
\end{align*}
We first deal with the sum of $S_1$: In this case, we assume $N_2\leqslant N_1\leqslant CN_2$, which does not contradict with the main assumption $N_1\sim N_2$. Write $N_1=N_2+RN_3$, where $R\geqslant0$. Then we can split the interval $[N_1,2N_1)$ into
\begin{align}\label{split-N_1}
\bigcup_{\alpha=R}^{R+J-1}I_\alpha=\bigcup_{\alpha=R}^{R+J-1}[N_2+\alpha N_3,N_2+(\alpha+1)N_3).
\end{align}
Similarly, we have
\begin{align}\label{split-N_2}
[N_2,2N_2)=\bigcup_{\beta=0}^{K-1}I_\beta=\bigcup_{\beta=0}^{J-1}[N_2+\beta N_3,N_2+(\beta+1)N_3).
\end{align}
In this case, we have $\alpha-\beta>C$.  Thanks to Lemma \ref{keylemma}, we can estimate $J_1$ by
\begin{align*}
&\hspace{2ex} \sum_{S_1}\int_{0}^{t}\int_{M}\left[1-\frac{m(n_1)}{m(n_2)m(n_3)m(n_4)}\right]\overline{\pi_{n_1}\phi_1}\pi_{n_2}\phi_2\overline{\pi_{n_3}\phi_3}\pi_{n_4}\phi_4dxdt\\
&=\sum_{S_1}\left[1-\frac{m(n_1)}{m(n_2)m(n_3)m(n_4)}\right]\frac{(-2)^l}{\big( n_1^2- n_2^2- n_3^2-n_4^2\big)^l}\\
&\hspace{3ex}\times\int_{0}^{t}\int_{M}\overline{\pi_{n_1}\phi_1}B(\pi_{n_2}\phi_2,\overline{\pi_{n_3}\phi_3},\pi_{n_4}\phi_4)+C(\pi_{n_2}\phi_2,\overline{\pi_{n_3}\phi_3},\pi_{n_4}\phi_4)\big),
\end{align*}
where
\begin{equation}\label{definition of tilde Bn1n2n3l}
B(f,g,h)=\mathcal{O}_{\substack{i+j+k=2l \\ 0\leqslant i,j,k \leqslant l}}\left(\nabla^i f * \nabla^j g * \nabla^k h \right)
\end{equation}

\begin{equation}\label{definition of tilde Cn1n2n3l}
C(f,g,h)=\mathcal{O}_{\substack{i+j+k\leq 2(l-1) \\ 0\leq i,j,k \leq l-1}}\left(\nabla^aR*\nabla^i f * \nabla^j g * \nabla^k h \right)
\end{equation}
and $R$ is the tensor comes from the Riemannian curvature tensor by contracting and differentiating $a$ times.
In fact, both $\tilde B$ and $\tilde C$ are trilinear operators that are linear combinations of products of differential operators on $M$ of the form $\widetilde{Q}_1(f)\widetilde{Q}_2(g)\widetilde{Q}_3(h)$ whose respective orders $i,j,k$ satisfy $i+j+k \leq 2l$ and $0\leq i,j,k \leq l$. The spetral localization formula we used gives us the large negative power of $n_k$, which can cancel the positive power obtained by the $I$-multipliers.  
\begin{align*}
&\hspace{5ex}\sum_{S_1}\int_{0}^{t}\int_{M}
\left[1-\frac{m(n_1)}{m(n_2)m(n_3)m(n_4)}\right]\overline{\pi_{n_1}(\phi_1)}\pi_{n_2}(\phi_2)\overline{\pi_{n_3}(\phi_3)}\pi_{n_4}(\phi_4) dxdt\\
&=\operatorname{l.c.}\sum_{S_1}
\left[1-\frac{m(n_1)}{m(n_2)m(n_3)m(n_4)}\right]\frac{(-2)^l}{(n_1^2-n_2^2-n_3^2-n_4^2)^l}\\
&\hspace{3ex}\times\int_0^t\int_{M}\overline{\pi_{n_1}(\phi_1)}\big[Q_2(\pi_{n_2}\phi_2)\overline{Q_3(\pi_{n_3}\phi_3)}
Q_4(\pi_{n_4}\phi_4)\big]dxdt,
\end{align*}
where $\operatorname{l.c.}$ is the abbreviation of linear combination of the term in  $\eqref{definition of tilde Bn1n2n3l}$ and $\eqref{definition of tilde Cn1n2n3l}$.

Let us denote:
\begin{equation}\label{def of bar m}
\bar m (n_1,n_2,n_3,n_4;l)=\left[1-\frac{m(n_1)}{m(n_2)m(n_3)m(n_4)}\right]\frac{(-2)^l}{(n_1^2-n_2^2-n_3^2-n_4^2)^l}.
\end{equation}
Then the sum of $S_1$ part can be bounded by the linear combination of
\begin{align}\label{reduce-S_1}
\sum_{S_1} \bar m(n_1,n_2,n_3,n_4;l) \int_0^t\int_{M}\overline{\pi_{n_1}(\phi_1)}
Q_2(\pi_{n_2}\phi_2)\overline{Q_3(\pi_{n_3}\phi_3)}Q_4(\pi_{n_4}\phi_4)dx\,dt.
\end{align}
$\eqref{reduce-S_1}$ can be regarded  as the  multilinear spectral multiplier operator
where $\overline{m}$ is the multiplier. Unfortunately, we cannot use Lemma  \ref{multiplier-lemma} directly, since $supp(\bar{m})$$\notin [0,1]^4$. In Case Two(2), the high frequency $N_1\sim N_2$ which is not separate, which leads us to break down the frequencies $n_1\in \cup I_\alpha$ and $n_2\in\cup I_\beta$. In the sum of $S_1$, $n_1\in I_\alpha$ is sufficiently separate from $n_2\in I_\beta$. This decomposition makes us to adapt the proof of Lemma \ref{multiplier-lemma} to every sub-intervals $I_\alpha$ and $I_\beta$. After normalization of $n_i$(using the fact $n_i\sim N_i$ and the splitting $\eqref{split-N_1}$ and $\eqref{split-N_2}$), we can write $\overline{m}(n_1,n_2,n_3,n_4)$ as a smooth function which is compactly supported. More precisely,
by $\eqref{split-N_1}$ and $\eqref{split-N_2}$, $n_1 \in I_{\alpha}=[N_2+\alpha N_3,N_2+(\alpha+1)N_3)$, $n_2 \in I_{\beta}=[N_2+\beta N_3,N_2+(\beta+1)N_3)$, $n_3 \sim N_3$ and $n_4 \sim N_4$ we can write
\begin{equation}\label{case 2 the tildas}
\begin{split}
	n_1&=N_2+ \alpha N_3 +N_3 \Tilde{n}_1 \text{ with $\Tilde{n}_1 \in [0,1)$}\\
	n_2&=N_2+ \beta N_3 +N_3 \Tilde{n}_2 \text{ with $\Tilde{n}_2 \in [0,1]$}\\
	n_3&=N_3(1+ \Tilde{n}_3) \text{ with $\Tilde{n}_3 \in [0,1]$}\\
	n_4&=N_4(1+ \Tilde{n}_4) \text{ with $\Tilde{n}_4 \in [0,1]$}.\\
\end{split}
\end{equation}
We now define the function $\Psi: [0,1]^4 \to \R$ given by:
\begin{equation}\label{def of psi step 1}
\Psi(\tilde{n}_1,\tilde{n}_2,\Tilde{n}_3,\tilde{n}_4)=\overline{m}(n_1,n_2,n_3,n_4)
\end{equation}
with $n_i,\Tilde{n}_i$ given in $\eqref{case 2 the tildas}$. Extend $\Psi$ to a $C^{\infty}$ compactly supported function on $[-2,2]^4$ and then as a function on $\Bbb{T}^4$. This allows us to express it in Fourier series:
$$
\Psi(\tilde n_1,\ldots,\tilde n_4)=\sum_{\theta_i \in \Bbb{Z}/4} e^{i (\theta_1 \tilde n_1+\ldots+\theta_4 \tilde n_4)}A(\theta_1,\ldots,\theta_4)
$$
where its Fourier coefficient is bounded by
$$
\sum_{\theta_i\in\Bbb{Z}/4} |A(\theta_1, \ldots,\theta_4)|\lesssim \sup_{\left|\beta\right|=s}\big\Vert\partial^\beta \Psi\Vert_{L^\infty([0,2]^4)},\quad\forall s>2.
$$
Using the above analysis , we can write $\eqref{reduce-S_1}$ as
\begin{align}\label{reduce-S_1-2}
\eqref{reduce-S_1}=&\sum_{(\alpha,\beta)\in S_1}\sum_{\substack{n_1 \in I_\alpha, n_2 \in I_\beta\\ n_3 \sim N_3,n_4 \sim N_4}}\overline{m}(n_1,n_2,n_3,n_4) \int_0^t \int_{M}\overline{\pi_{n_1} \phi_1}Q_2( \pi_{n_2}\phi_2)\overline{Q_3( \pi_{n_3}\phi_3)}Q_4( \pi_{n_4} \phi_4)dxdt\notag\\
=& \sum_{(\alpha,\beta)\in S_1}\sum_{\substack{n_1 \in I_\alpha, n_2 \in I_\beta\\ n_3 \sim N_3,n_4 \sim N_4}}\sum_{\theta_i \in \Bbb{Z}}A(\theta_1,\ldots,\theta_4) \notag\\
&\times\int_0^t \int_{M}e^{i\theta_1 \tilde n_1}\overline{\pi_{n_1} \phi_1}Q_2( e^{i\theta_2 \tilde n_2}\pi_{n_2}\phi_2)\overline{Q_3( e^{i\theta_3 \tilde n_3}\pi_{n_3}\phi_3)}Q_4( e^{i\theta_4 \tilde n_4}\pi_{n_4}\phi_4)dxdt\notag\\
=&\sum_{(\alpha,\beta)\in S_1} \sum_{\theta_i \in \Bbb{Z}/4}A(\theta_1,\ldots,\theta_4)\int_0^t \int_{M}
\overline{P_{I_\alpha} \phi^{\theta_1}_1}
Q_2(P_{I_\beta}\phi^{\theta_2}_2)\overline{Q_3( \phi^{\theta_3}_3)}
Q_4(\phi^{\theta_4}_4))dxdt,
\end{align}
where
\begin{equation}\label{the tilda phis}
\begin{split}
	\phi_1^{\theta_1}&=\sum_{n_1}e^{-i\theta_1 (n_1-(N_2+\alpha N_3))\over N_3}\pi_{n_1}\phi_1,\\
	\phi_2^{\theta_2}&=\sum_{n_2}e^{i\theta_2 (n_2-(N_2+\beta N_3))\over N_3}\pi_{n_2}\phi_2,\\
	\phi_3^{\theta_3}&=\sum_{n_3}e^{-i\theta_3 (n_3-N_3)\over N_3}\pi_{n_3}\phi_3,\\
	\phi_4^{\theta_4}&=\sum_{n_4}e^{i\theta_4 (n_4-N_4)\over N_4}\pi_{n_4}\phi_4.\\
\end{split}
\end{equation}
To prove that $A(\theta_1,\cdots,\theta_4)$ is $\ell^1$ summable, we need to bound $\Phi$ by
\begin{align}\label{part-1}
\Vert\Psi\Vert_{C^2([0,1]^4)}\lesssim(\alpha-\beta)\frac{N_3}{N_2}\frac{1}{(N_2N_3(\alpha-\beta)^l)}.
\end{align}
This estimate can be proved by using the similar strategy which  is very closed to that in Hani's paper\cite{Hani2}.

The rest of the right-hand side of $\eqref{reduce-S_1-2}$ is
\begin{align}\label{rest-reduce}
\int_{0}^{t}\int_{M}\overline{P_{I_\alpha}\phi_1^{\theta_1}}Q_2(P_{I_\beta}\phi_2^{\theta_2})\overline{Q_3(\phi_3^{\theta_3})}Q_4(\phi_4^{\theta_4})dxdt.
\end{align}
To estimate $\eqref{rest-reduce}$, we need to use  bilinear estimates $\eqref{bil-PD-zoll-X}$ and $\eqref{bil-PD-pro-X}$ to estimate when $M$ is Zoll manifold and $M=\Bbb{S}^2\times\Bbb{S}^1$ respectively,
\begin{gather*}
  \bigg\Vert P(D) e^{it\Delta_g}u_0Q(D) e^{it\Delta_g}v_0\bigg\Vert_{L^2([0,1]\times M)}\lesssim\begin{cases} N_1^m  N_2^nN_2^{\frac{1}{2}+}\Vert u_0\Vert_{L^2}\Vert v_0\Vert_{L^2},&\mbox{when }$M$ \mbox{ is Zoll manifold},\\
   N_1^m  N_2^nN_2^{\frac{3}{4}+}\Vert u_0\Vert_{L^2}\Vert v_0\Vert_{L^2},&\mbox{when }M=\Bbb{S}^2\times\Bbb{S}^1.
   \end{cases}
\end{gather*}
where $P(D)$ and $Q(D)$ are two differential operators with order $m$ and $n$ respectively. These estimates can be proved by using bilinear estimates $\eqref{bil-interpo-pro}$ and $\eqref{bil-X}$ and Bernstein's estimates since $u_0$ and $v_0$ are spectrally localized function. 
 When $M$ is Zoll manifold, we have
\begin{align}\label{part-2}
\eqref{rest-reduce}&\lesssim\Big\Vert P_{I_\alpha}\phi_1^{\theta_1} Q_4(\phi_4^{\theta_4})\Big\Vert_{L^2_{t,x}([0,t]\times M)}\Big\Vert Q_2(P_{I_\beta}\phi_2^{\theta_2}) Q_3(\phi_3^{\theta_3})\Big\Vert_{L^2_{t,x}([0,t]\times M)}\notag\\
&\lesssim  \delta^{\frac{1}{4}-}N_4^{l_1} N_2^{l_2} N_3^{l_3}N_3^{\frac{1}{2}+}N_4^{1-}\Big\Vert P_{I_\alpha}\phi_1\Big\Vert_{X^{0,b(1-)}}\Vert \phi_4\Vert_{X^{0,b(1-)}} \Big\Vert P_{I_\beta}\phi_2\Big\Vert_{X^{0,1/2+}}\Big\Vert\phi_3\Big\Vert_{X^{0,1/2+}}\notag\\
&\lesssim \delta^{\frac{1}{4}-}N_2^l N_3^l N_3^{\frac{1}{2}+}N_4^{1-}
\times\Big\Vert P_{I_\alpha}\phi_1\Big\Vert_{X^{0,1/2+}} \Big\Vert P_{I_\beta}\phi_2\Big\Vert_{X^{0,1/2+}}\prod_{j=3}^{4}\Big\Vert\phi_j\Big\Vert_{X^{0,1/2+}}
\end{align}
where we use  $\eqref{int-bil}$ in the second inequality with $\deg(Q_i)=l_i$, $i=2,3,4$ and the fact $\deg(Q_2)+\deg(Q_3)+\deg(Q_4)\leqslant2l$, $0\leqslant\deg(Q_i)\leqslant l$.
Similarly, when $M=\Bbb{S}^2\times\Bbb{S}^1$, we can obtain
\begin{align}\label{part-2pro}
\eqref{rest-reduce}&\lesssim\Big\Vert P_{I_\alpha}\phi_1^{\theta_1} Q_4(\phi_4^{\theta_4})\Big\Vert_{L^2_{t,x}([0,t]\times M)}\Big\Vert Q_2(P_{I_\beta}\phi_2^{\theta_2}) Q_3(\phi_3^{\theta_3})\Big\Vert_{L^2_{t,x}([0,t]\times M)}\notag\\
&\lesssim  \delta^{\frac{1}{6}-}N_4^{l_1} N_2^{l_2} N_3^{l_3}N_3^{\frac{3}{4}+}N_4^{1-}\Big\Vert P_{I_\alpha}\phi_1\Big\Vert_{X^{0,b(1-)}}\Vert \phi_4\Vert_{X^{0,b(1-)}} \Big\Vert P_{I_\beta}\phi_2\Big\Vert_{X^{0,1/2+}}\Big\Vert\phi_3\Big\Vert_{X^{0,1/2+}}\notag\\
&\lesssim N_2^l N_3^l \delta^{\frac{1}{6}-}N_3^{\frac{3}{4}+}N_4^{1-}
\times\Big\Vert P_{I_\alpha}\phi_1\Big\Vert_{X^{0,1/2+}} \Big\Vert P_{I_\beta}\phi_2\Big\Vert_{X^{0,1/2+}}\prod_{j=3}^{4}\Big\Vert\phi_j\Big\Vert_{X^{0,1/2+}}.
\end{align}

Combining $\eqref{part-1}$ and $\eqref{part-2}$,  for the case of Zoll manifold, one gets
\begin{align*}
\eqref{reduce-S_1}&\lesssim\sum_{\alpha,\beta}\bigg(\sum_{\theta_i \in \Bbb{Z}/4} \left|A(\theta_1,\ldots,\theta_4)\right|\bigg)(N_2N_3)^lN_3^{\frac{1}{2}+}N_4^{1-}\delta^{\frac{1}{4}-}
||P_{I_\alpha}\phi_1||_{X^{0,1/2+}} ||P_{I_\beta}\phi_2||_{X^{0,1/2+}}\prod_{i=3}^4||\phi_i||_{X^{0,1/2+}}\\
\lesssim&\sum_{\alpha,\beta} |\alpha-\beta|\frac{N_3}{N_2}\frac{(N_2N_3)^l}{\left(N_2N_3(\alpha-\beta)\right)^l}N_3^{\frac{1}{2}+}N_4^{1-}\delta^{\frac{1}{4}-}
||P_{I_\alpha}\phi_1||_{X^{0,1/2+}} ||P_{I_\beta}\phi_2||_{X^{0,1/2+}}\prod_{i=3}^4||\phi_i||_{X^{0,1/2+}}\\
\lesssim& \frac{N_3}{N_2}N_3^{\frac{1}{2}+}N_4^{1-}\delta^{\frac{1}{4}-}\sum_{\alpha,\beta} \frac{1}{\left(\alpha-\beta\right)^{l-1}}
||P_{I_\alpha}\phi_1||_{X^{0,1/2+}} ||P_{I_\beta}\phi_2||_{X^{0,1/2+}}\prod_{i=3}^4||\phi_4||_{X^{0,1/2+}}.
\end{align*}

For the case of $\Bbb{S}^2\times\Bbb{S}^1$, one can obtain
\begin{align*}
\eqref{reduce-S_1}&\lesssim\sum_{\alpha,\beta}\bigg(\sum_{\theta_i \in \Bbb{Z}/4} \left|A(\theta_1,\ldots,\theta_4)\right|\bigg)(N_2N_3)^lN_3^{\frac{3}{4}+}N_4^{1-}\delta^{\frac{1}{6}-}
||P_{I_\alpha}\phi_1||_{X^{0,1/2+}} ||P_{I_\beta}\phi_2||_{X^{0,1/2+}}\prod_{i=3}^4||\phi_i||_{X^{0,1/2+}}\\
\lesssim&\sum_{\alpha,\beta} |\alpha-\beta|\frac{N_3}{N_2}\frac{(N_2N_3)^l}{\left(N_2N_3(\alpha-\beta)\right)^l}N_3^{\frac{3}{4}+}N_4^{1-}\delta^{\frac{1}{6}-}
||P_{I_\alpha}\phi_1||_{X^{0,1/2+}} ||P_{I_\beta}\phi_2||_{X^{0,1/2+}}\prod_{i=3}^4||\phi_i||_{X^{0,1/2+}}\\
\lesssim& \frac{N_3}{N_2}N_3^{\frac{3}{4}+}N_4^{1-}\delta^{\frac{1}{6}-}\sum_{\alpha,\beta} \frac{1}{\left(\alpha-\beta\right)^{l-1}}
||P_{I_\alpha}\phi_1||_{X^{0,1/2+}} ||P_{I_\beta}\phi_2||_{X^{0,1/2+}}\prod_{i=3}^4||\phi_4||_{X^{0,1/2+}}.
\end{align*}

Taking $l\geq 100$ and applying Schur's test lemma, we have the following estimate for the case of Zoll manifold,
\begin{align}\label{bound on S_1}
\eqref{reduce-S_1}&\lesssim \frac{N_3}{N_2}N_3^{\frac{1}{2}+}N_4^{1-}\delta^{\frac{1}{4}-}\prod_{i=1}^4||\phi_i||_{X^{0,1/2+}}\notag\\
&\lesssim\frac{N_3}{N_2}N_3^{\frac{1}{2}+}N_4^{1-}\delta^{\frac{1}{4}-}\frac{N_1}{N_2N_3N_4}||\phi_1||_{X^{-1,\frac{1}{2}+}}\prod_{j=2}^{4}||\phi_j||_{X^{1,\frac{1}{2}+}}\\
&\lesssim\frac{1}{ N^{\frac{1}{2}-}}\delta^{\frac{1}{4}-}(N_1N_2N_3N_4)^{0-}||\phi_1||_{X^{-1,\frac{1}{2}+}}\prod_{j=2}^{4}||\phi_j||_{X^{1,\frac{1}{2}+}}.
\end{align}
Hence, we obtain the estimate of $\eqref{Goal1}$ for Zoll manifold.

For the case of $\Bbb{S}^2\times\Bbb{S}^1$, we have
\begin{align}\label{bound on S_1-pro}
\eqref{reduce-S_1}&\lesssim \frac{N_3}{N_2}N_3^{\frac{3}{4}+}N_4^{1-}\delta^{\frac{1}{6}-}\prod_{i=1}^4||\phi_i||_{X^{0,1/2+}}\notag\\
&\lesssim\frac{N_3}{N_2}N_3^{\frac{3}{4}+}N_4^{1-}\delta^{\frac{1}{6}-}\frac{N_1}{N_2N_3N_4}||\phi_1||_{X^{-1,\frac{1}{2}+}}\prod_{j=2}^{4}||\phi_j||_{X^{1,\frac{1}{2}+}}\\
&\lesssim\frac{1}{ N^{\frac{1}{4}-}}\delta^{\frac{1}{6}-}(N_1N_2N_3N_4)^{0-}||\phi_1||_{X^{-1,\frac{1}{2}+}}\prod_{j=2}^{4}||\phi_j||_{X^{1,\frac{1}{2}+}}.
\end{align}
Hence, we obtain the estimate of $\eqref{Goal1}$ for the product of spheres.

The estimate of $S_2$ is similar with the previous case where one gets the same bound for $S_2$ just by interchanging the roles of $N_1$ and $N_2$ above. The bound for $S_3=\{n_1\in I_\alpha,\hspace{1ex}n_2\in  I_\beta, \left|n_1-n_2\right|\leqslant8N_3\}$ (near diagonal terms) is simpler. For each $I_\beta$, let $S_3^\beta$ be the set of $I_\alpha$ intervals that are at a distance less than $\leq 8N_3$ from $I_\beta$. Clearly there are finite elements in $S_3^\beta$:
\begin{align*}
&\sum_{S_3} \int_0^t\int_{M}
\left[1-\frac{m(n_1)}{m(n_2)m(n_3)m(n_4)}\right]\overline{\pi_{n_1}(\phi_1)}\pi_{n_2}(\phi_2)\overline{\pi_{n_3}(\phi_3)}\pi_{n_4}(\phi_4) dx\,dt\\
&=\sum_{I_\beta}\sum_{I_\alpha\in S_3^\beta} \sum_{\substack{n_1 \in I_\alpha, n_2 \in I_\beta\\ n_3 \sim N_3,n_4 \sim N_4}} \int_0^t\int_{M}
\left[1-\frac{m(n_1)}{m(n_2)m(n_3)m(n_4)}\right]\overline{\pi_{n_1}(\phi_1)}\pi_{n_2}(\phi_2)\overline{\pi_{n_3}(\phi_3)}\pi_{n_4}(\phi_4) dx\,dt.
\end{align*}
In this case, we will use Lemma \ref{multiplier-lemma} to estimate the sum
$$
\sum_{\substack{n_1 \in I_\alpha, n_2 \in I_\beta\\ n_3 \sim N_3,n_4 \sim N_4}} \int_0^t\int_{M}
\left[1-\frac{m(n_1)}{m(n_2)m(n_3)m(n_4)}\right]\overline{\pi_{n_1}(\phi_1)}\pi_{n_2}(\phi_2)\overline{\pi_{n_3}(\phi_3)}\pi_{n_4}(\phi_4) dx\,dt.
$$
By Lemma \ref{multiplier-lemma}, we need to estimate $m(n_1,n_2,n_3,n_4)$ and
\begin{align}\label{mul-2}
\int_{0}^{t}\int_{M}\overline{\pi_{n_1}(\phi_1)}\pi_{n_2}(\phi_2)\overline{\pi_{n_3}(\phi_3)}\pi_{n_4}(\phi_4) dx\,dt.
\end{align}
By the analysis in $S_1$, we note that the multiplier
\begin{align*}
\overline{m}(n_1,n_2,n_3,n_4)=\frac{N_2}{N_3}\left[1-\frac{m(n_1)}{m(n_2)m(n_3)m(n_4)}\right]
\end{align*}
satisfies the condition $\eqref{mult}$. Then using bilinear estimate $\eqref{int-bil}$, we get
\begin{align*}
\left|\int_{0}^{t}\int_{M}f_1(t,x)f_2(t,x)f_3(t,x)f_4(t,x) dx\,dt\right|&\lesssim\big\Vert f_1f_4\big\Vert_{L_{t,x}^2}\big\Vert f_2f_3\big\Vert_{L_{t,x}^2}\\
&\lesssim N_3^{\frac{1}{2}+}N_4^{1-}\delta^{\frac{1}{4}-}\prod_{i=1}^{4}\Vert f_i\Vert_{X^{0,\frac{1}{2}+}},
\end{align*}
where $f_i(t,x)$, $j=1,2,3,4$ are frequency localized localization and the frequencies of $f_j$ also satisfy the assumption of this case.
Similarly, we have
\begin{align*}
\left|\int_{0}^{t}\int_{M}f_1(t,x)f_2(t,x)f_3(t,x)f_4(t,x) dx\,dt\right|\lesssim N_3^{\frac{3}{4}+}N_4^{1-}\delta^{\frac{1}{6}-}\prod_{i=1}^{4}\Vert f_i\Vert_{X^{0,\frac{1}{2}+}},
\end{align*}
where $M=\Bbb{S}^2\times\Bbb{S}^1$.

Therefore, for the case of Zoll manifold, we obtain
\begin{align*}
& \left|\sum_{S_3}\int_{0}^{t}\int_{M}\left[1-\frac{m(n_1)}{m(n_2)m(n_3)m(n_4)}\right]\overline{\pi_{n_1}\phi_1}\pi_{n_2}\phi_2\overline{\pi_{n_3}\phi_3}\pi_{n_4}\phi_4dx\,dt\right|\\
&\lesssim\sum_{I_\beta}\sum_{I_\alpha\in S_3^\beta}\frac{N_3}{N_2}N_3^{\frac{1}{2}+}N_4^{1-}\delta^{\frac{1}{4}-}\big\Vert P_{I_\alpha}\phi_1\big\Vert_{X^{0,\frac{1}{2}+}}\big\Vert P_{I_\beta}\phi_2\big\Vert_{X^{0,\frac{1}{2}+}}\prod_{j=3}^{4}\Vert\phi_i\Vert_{X^{0,\frac{1}{2}+}}\\
&\lesssim\frac{N_3}{N_2}N_3^{\frac{1}{2}+}N_4^{1-}\delta^{\frac{1}{4}-}\prod_{j=3}^{4}\Vert\phi_i\Vert_{X^{0,\frac{1}{2}+}}\left(\sum_{I_\alpha\in S_3^\beta}\big\Vert P_{I_\alpha}\phi_1\big\Vert^2_{X^{0,\frac{1}{2}+}}\left(\sum_{I_\beta}\big\Vert P_{I_\beta}\phi_2\big\Vert_{X^{0,\frac{1}{2}+}}\right)^2\right)^{\frac{1}{2}}\\
&\lesssim\frac{N_3}{N_2}N_3^{\frac{1}{2}+}N_4^{1-}\delta^{\frac{1}{4}-}\prod_{j=1}^{4}\Vert \phi_j\Vert_{X^{0,\frac{1}{2}+}},
\end{align*}
where we use the fact $|S_3^\beta|\leqslant C$ and the Cauchy-Schwarz ineqaulity in the last two inequalities.

For the case of product of spheres, we obtain
\begin{align*}
& \left|\sum_{S_3}\int_{0}^{t}\int_{M}\left[1-\frac{m(n_1)}{m(n_2)m(n_3)m(n_4)}\right]\overline{\pi_{n_1}\phi_1}\pi_{n_2}\phi_2\overline{\pi_{n_3}\phi_3}\pi_{n_4}\phi_4dx\,dt\right|\\
&\lesssim\frac{N_3}{N_2}N_3^{\frac{3}{4}+}N_4^{1-}\delta^{\frac{1}{6}-}\prod_{j=3}^{4}\Vert\phi_i\Vert_{X^{0,\frac{1}{2}+}}\left(\sum_{I_\alpha\in S_3^\beta}\big\Vert P_{I_\alpha}\phi_1\big\Vert^2_{X^{0,\frac{1}{2}+}}\left(\sum_{I_\beta}\big\Vert P_{I_\beta}\phi_2\big\Vert_{X^{0,\frac{1}{2}+}}\right)^2\right)^{\frac{1}{2}}\\
&\lesssim\frac{N_3}{N_2}N_3^{\frac{3}{4}+}N_4^{1-}\delta^{\frac{1}{6}-}\prod_{j=1}^{4}\Vert \phi_j\Vert_{X^{0,\frac{1}{2}+}}. 
\end{align*}
Using the spectrally localized property, we complete the estimate of $\eqref{Goal1}$,
\begin{align*}
\mbox{(LHS) of }\eqref{Goal1}&\lesssim\frac{N_3}{N_2}N_3^{\frac{1}{2}+}N_4^{1-}\delta^{\frac{1}{4}-}\frac{N_1}{N_2N_3N_4}\Vert \phi_1\Vert_{X^{-\frac{\sigma}{2},\frac{1}{2}+}}\prod_{j=2}^{4}\Vert\phi_j\Vert_{X^{\frac{\sigma}{2},\frac{1}{2}+}}\\
&\lesssim\delta^{\frac{1}{4}-}\frac{1}{ N^{\frac{1}{2}-}}N_2^{0-}\Vert\phi_1\Vert_{X^{-1,\frac{1}{2}+}}\prod_{j=2}^{4}\Vert\phi_j\Vert_X^{1,\frac{1}{2}+}(\mbox{Zoll manifold}),\\
\mbox{(LHS) of }\eqref{Goal1}&\lesssim \delta^{\frac{1}{6}-}\frac{N_3}{N_2}N_3^{\frac{3}{4}+}N_4^{1-}\frac{N_1}{N_2N_3N_4}\Vert \phi_1\Vert_{X^{-\frac{\sigma}{2},\frac{1}{2}+}}\prod_{j=2}^{4}\Vert\phi_j\Vert_{X^{\frac{\sigma}{2},\frac{1}{2}+}}\\
&\lesssim\delta^{\frac{1}{6}-}\frac{1}{ N^{\frac{1}{4}-}}N_2^{0-}\Vert\phi_1\Vert_{X^{-1,\frac{1}{2}+}}\prod_{j=2}^{4}\Vert\phi_j\Vert_X^{1,\frac{1}{2}+}(\mbox{$\Bbb{S}^2\times\Bbb{S}^1$}).\\
\end{align*}

\textbf{Case Three: $N_2 \sim N_3 \gtrsim N$}.

In this case, we can write
\begin{align*}
  \frac{m(N_1)}{m(N_2)m(N_3)m(N_4)}\lesssim \frac{m(N_1)}{m(N_4)}(N^{-2}N_2N_3)^{1-s},
\end{align*}
where
\begin{align}\label{case3}
  \frac{m(N_0)}{m(N_4)}\lesssim\begin{cases}
                                 1, & \mbox{if } N_1,N_4\ll N \\
                                 \big(\frac{N}{N_1}\big)^{1-s}, & \mbox{if } N_4\lesssim N\lesssim N_1 \\
                                 \big(\frac{N_4}{N}\big)^{1-s}, & \mbox{if } N_1\lesssim N\lesssim N_4 \\
                                 \big(\frac{N_4}{N_1}\big)^{1-s}, & \mbox{if }N\lesssim N_1, N_4.
                               \end{cases}
\end{align}
Taking the first term in $\eqref{case3}$, then using bilinear estimate $\eqref{int-bil}$, $\eqref{bil-interpo-pro}$ and Bernstein's inequality, we obtain
\begin{align*}
  \mbox{(LHS) of }\eqref{Goal1}&\lesssim(N^{-2}N_2N_3)^{1-s}\big\Vert \pi_{n_1}\phi_1\pi_{n_3}\phi_3\big\Vert_{L_{t,x}^2([0,\delta]\times M)}\big\Vert \pi_{n_2}\phi_2\pi_{n_4}\phi_4\big\Vert_{L_{t,x}^2([0,\delta]\times M)}\\
  &\lesssim N^{-2(1-s)}\frac{(N_2N_3)^{1-s}N_1^{\frac{1}{2}+}N_4^{1-}N_1}{N_2N_3N_4}\delta^{\frac{1}{4}-}\Vert \phi_1\Vert_{X^{-1,\frac{1}{2}+}}\prod_{j=2}^{4}\Vert \phi_j\Vert_{X^{1,\frac{1}{2}+}}\\
  &\lesssim \delta^{\frac{1}{4}-}N^{-\frac{1}{2}+}\Vert \phi_1\Vert_{X^{-1,\frac{1}{2}+}}\prod_{j=2}^{4}\Vert \phi_j\Vert_{X^{1,\frac{1}{2}+}} (\mbox{ Zoll manifold })
\end{align*} 
and
\begin{align*}
  \mbox{(LHS) of }\eqref{Goal1}&\lesssim(N^{-2}N_2N_3)^{1-s}\big\Vert \pi_{n_1}\phi_1\pi_{n_3}\phi_3\big\Vert_{L_{t,x}^2([0,\delta]\times M)}\big\Vert \pi_{n_2}\phi_2\pi_{n_4}\phi_4\big\Vert_{L_{t,x}^2([0,\delta]\times M)}\\
  &\lesssim N^{-2(1-s)}\frac{(N_2N_3)^{1-s}N_1^{\frac{3}{4}+}N_4^{1-}N_1}{N_2N_3N_4}\delta^{\frac{1}{6}-}\Vert \phi_1\Vert_{X^{-1,\frac{1}{2}+}}\prod_{j=2}^{4}\Vert \phi_j\Vert_{X^{1,\frac{1}{2}+}}\\
  &\lesssim \delta^{\frac{1}{6}-}N^{-\frac{1}{4}+}\Vert \phi_1\Vert_{X^{-1,\frac{1}{2}+}}\prod_{j=2}^{4}\Vert \phi_j\Vert_{X^{1,\frac{1}{2}+}} \hspace{1ex}(\Bbb{S}^2\times\Bbb{S}^1).
\end{align*}
For the case of $M=\Bbb{S}^2\times\Bbb{S}^1$, to obtain the negative power of $N_2$, we need $s>\frac{7}{8}$.
The second term in $\eqref{case3}$ can be treated similarly and the same bound also holds. We omit the details here.

Now we take the third term in $\eqref{case3}$. Utilizing $\eqref{int-bil}$ and $\eqref{bil-interpo-pro}$ and H\"{o}lder's inequality, we have
\begin{align*}
  \mbox{(LHS) of }\eqref{Goal1}&\lesssim(N^{-2}N_2N_3)^{1-s}\big(\frac{N_4}{N}\big)^{1-s}\big\Vert \pi_{n_1}\phi_1\pi_{n_3}\phi_3\big\Vert_{L_{t,x}^2([0,\delta]\times M)}\big\Vert \pi_{n_2}\phi_2\pi_{n_4}\phi_4\big\Vert_{L_{t,x}^2([0,\delta]\times M)}\\
  &\lesssim(N^{-2}N_2N_3)^{1-s}\big(\frac{N_4}{N}\big)^{1-s}(N_1N_4)^{\frac{1}{2}+}\frac{N_1}{N_2N_3N_4}\Vert \phi_1\Vert_{X^{-1,\frac{1}{2}+}}\prod_{j=2}^{4}\Vert \phi_j\Vert_{X^{1,\frac{1}{2}+}}\\
  &\lesssim N^{-1+}N_1^{0-}N_2^{0-}\Vert \phi_1\Vert_{X^{-1,\frac{1}{2}+}}\prod_{j=2}^{4}\Vert \phi_j\Vert_{X^{1,\frac{1}{2}+}}\hspace{1ex}(\mbox{Zoll manifold})
\end{align*}
and
\begin{align*}
  \mbox{(LHS) of }\eqref{Goal1}&\lesssim(N^{-2}N_2N_3)^{1-s}\big(\frac{N_4}{N}\big)^{1-s}\big\Vert \pi_{n_1}\phi_1\pi_{n_3}\phi_3\big\Vert_{L_{t,x}^2([0,\delta]\times M)}\big\Vert \pi_{n_2}\phi_2\pi_{n_4}\phi_4\big\Vert_{L_{t,x}^2([0,\delta]\times M)}\\
  &\lesssim(N^{-2}N_2N_3)^{1-s}\big(\frac{N_4}{N}\big)^{1-s}(N_1N_4)^{\frac{3}{4}+}\frac{N_1}{N_2N_3N_4}\Vert \phi_1\Vert_{X^{-1,\frac{1}{2}+}}\prod_{j=2}^{4}\Vert \phi_j\Vert_{X^{1,\frac{1}{2}+}}\\
  &\lesssim N^{-\frac{1}{2}+}N_1^{0-}N_2^{0-}\Vert \phi_1\Vert_{X^{-1,\frac{1}{2}+}}\prod_{j=2}^{4}\Vert \phi_j\Vert_{X^{1,\frac{1}{2}+}}\hspace{1ex}(\Bbb{S}^2\times\Bbb{S}^1).
\end{align*}
The same bound also holds for the last term.

\textbf{Case Four}: $N_1\sim N_2\gg N_3\gtrsim N$.
Recall that in this case $N_1 \sim N_2$. First, we consider the case when $m(N_4)\sim1$, i.e. $N_1\sim N_2\gtrsim N_3\gtrsim N\gg N_4$. Using the same strategy compared to the previous case, for the Zoll manifold case,  we have 
\begin{align}\label{Case 4 split}
	\operatorname{L.H.S. of }\eqref{Goal1}\lesssim&
	\frac{m(N_1)}{m(N_2)m(N_3)m(N_4)} N_3^{\frac{1}{2}+}N_4^{1-}\frac{N_1}{N_2N_3N_4}\delta^{\frac{1}{4}-}||\phi_1||_{X^{-1,\frac{1}{2}+}}\prod_{i=2}^4||\phi_i||_{X^{1,1/2+}}\notag\\
&\lesssim N^{-(1-s)}N_3^{\frac{1}{2}-s+}N_4^{0-}\frac{N_1}{N_2}\delta^{\frac{1}{4}-}\big\Vert\phi_1\Vert_{X^{-1,\frac{1}{2}+}}\prod_{i=2}^4\big\Vert\phi_i\big\Vert_{X^{1,\frac{1}{2}+}}\notag\\
&\lesssim N^{-\frac{1}{2}+}\delta^{\frac{1}{4}-}\big\Vert\phi_1\big\Vert_{X^{-1,\frac{1}{2}+}}\prod_{i=2}^4\big\Vert\phi_i\big\Vert_{X^{1,\frac{1}{2}
+}},
\end{align}
while for the product of spheres  $\Bbb{S}^2\times\Bbb{S}^1$
\begin{align*}
  	\operatorname{L.H.S. of }\eqref{Goal1}\lesssim&
	\frac{m(N_1)}{m(N_2)m(N_3)m(N_4)} N_3^{\frac{3}{4}+}N_4^{1-}\frac{N_1}{N_2N_3N_4}\delta^{\frac{1}{6}-}||\phi_1||_{X^{-1,\frac{1}{2}+}}\prod_{i=2}^4||\phi_i||_{X^{1,1/2+}}\notag\\
&\lesssim N^{-(1-s)}N_3^{\frac{3}{4}-s+}N_4^{0-}\frac{N_1}{N_2}\delta^{\frac{1}{6}-}\big\Vert\phi_1\Vert_{X^{-1,\frac{1}{2}+}}\prod_{i=2}^4\big\Vert\phi_i\big\Vert_{X^{1,\frac{1}{2}+}}\notag\\
&\lesssim N^{-\frac{1}{4}+}\delta^{\frac{1}{6}-}\big\Vert\phi_1\big\Vert_{X^{-1,\frac{1}{2}+}}\prod_{i=2}^4\big\Vert\phi_i\big\Vert_{X^{1,\frac{1}{2}
+}}.
\end{align*}
Next, we consider the case when $N_1\sim N_2\gtrsim N_3\gtrsim  N_4\gtrsim N$. 
\begin{align}\label{Case 4 split-2}
	\operatorname{L.H.S. of }\eqref{Goal1}\lesssim&
	\frac{m(N_1)}{m(N_2)m(N_3)m(N_4)} (N_3N_4)^{\frac{1}{2}+}\frac{N_1}{N_2N_3N_4}||\phi_1||_{X^{-1,\frac{1}{2}+}}\prod_{i=2}^4||\phi_i||_{X^{1,1/2+}}\notag\\
&\lesssim \frac{1}{N^{2(1-s)}(N_3N_4)^{s-1}}(N_3N_4)^{\frac{1}{2}+}\frac{N_1}{N_2N_3N_4}\big\Vert\phi_1\big\Vert_{X^{-1,\frac{1}{2}+}}\prod_{i=2}^4\big\Vert\phi_i\big\Vert_{X^{1,\frac{1}{2}
+}}\notag\\
&\lesssim N^{-1+}\big\Vert\phi_1\big\Vert_{X^{-1,\frac{1}{2}+}}\prod_{i=2}^4\big\Vert\phi_i\big\Vert_{X^{1,\frac{1}{2}
+}}(\mbox{Zoll manifold})
\end{align}
and
\begin{align}\label{Case 4 split-2-pro}
	\operatorname{L.H.S. of }\eqref{Goal1}\lesssim&
	\frac{m(N_1)}{m(N_2)m(N_3)m(N_4)} (N_3N_4)^{\frac{3}{4}+}\frac{N_1}{N_2N_3N_4}||\phi_1||_{X^{-1,\frac{1}{2}+}}\prod_{i=2}^4||\phi_i||_{X^{1,1/2+}}\notag\\
&\lesssim \frac{1}{N^{2(1-s)}(N_3N_4)^{s-1}}(N_3N_4)^{\frac{3}{4}+}\frac{N_1}{N_2N_3N_4}\big\Vert\phi_1\big\Vert_{X^{-1,\frac{1}{2}+}}\prod_{i=2}^4\big\Vert\phi_i\big\Vert_{X^{1,\frac{1}{2}
+}}\notag\\
&\lesssim N^{-\frac{1}{2}+}\big\Vert\phi_1\big\Vert_{X^{-1,\frac{1}{2}+}}\prod_{i=2}^4\big\Vert\phi_i\big\Vert_{X^{1,\frac{1}{2}
+}}\quad(\Bbb{S}^2\times\Bbb{S}^1).
\end{align}
This finishes the proof of $\eqref{Goal1}$ and by the standard argument we can complete the proof of $\operatorname{Term_1}$. We summary the $\operatorname{Term_1}$ as follows
\begin{align*}
  &\left|\sum_{n_i\sim  N_i}\int_{0}^{\delta}\int_{M}\big(1-\frac{m(n_1)}{m(n_2)m(n_3)m(n_4)}\big)\pi_{n_1}\phi_1\pi_{n_2}\phi_2\pi_{n_3}\phi_3\pi_{n_4}\phi_4dxdt\right|\\
  &\lesssim \big(N^{-1+}+N^{-\frac{1}{2}+}\delta^{\frac{1}{4}-}\big)\big\Vert\phi_1\big\Vert_{X^{-1,\frac{1}{2}+}}\prod_{i=2}^4\big\Vert\phi_i\big\Vert_{X^{1,\frac{1}{2}
+}}, \hspace{2ex}(\mbox{Zoll manifold})
\end{align*}
and
\begin{align*}
  &\left|\sum_{n_i\sim  N_i}\int_{0}^{\delta}\int_{M}\big(1-\frac{m(n_1)}{m(n_2)m(n_3)m(n_4)}\big)\pi_{n_1}\phi_1\pi_{n_2}\phi_2\pi_{n_3}\phi_3\pi_{n_4}\phi_4dxdt\right|\\
  &\lesssim \big(N^{-\frac{1}{2}+}+N^{-\frac{1}{4}+}\delta^{\frac{1}{6}-}\big)\big\Vert\phi_1\big\Vert_{X^{-1,\frac{1}{2}+}}\prod_{i=2}^4\big\Vert\phi_i\big\Vert_{X^{1,\frac{1}{2}
+}}, \hspace{2ex}(\Bbb{S}^2\times\Bbb{S}^1).
\end{align*}

\subsubsection{Bound on $I_2$:}

Now we focus on the estimate of $I_2$. We will show that the decay estimate of $I_2$ is much better than that of $I_1$. Indeed, it can be demonstrated  that
\begin{equation}\label{decay of term 2}
I_2\lesssim \begin{cases}
\frac{1}{N^{\frac{3}{2}-}}\Vert Iu\Vert_{X^{1,1/2+}}^6,&\quad\mbox{When }M\mbox{ is Zoll manifold},\\
 \frac{1}{N^{\frac{5}{4}-}}\Vert Iu\Vert_{X^{1,1/2+}}^6,&\quad\mbox{When }M=\Bbb{S}^2\times\Bbb{S}^1.
 \end{cases}
\end{equation}
We will rewrite $I_2$ in a form that is more convenient to do a multilinear analysis in $L^p$ spaces:
\begin{align*}
I_2=&\sum_{n_i}\int_0^t\int_{M}
\left[1-\frac{m(n_1)}{m(n_2)m(n_3)m(n_4)}\right]\overline{\pi_{n_1}\left(I_N\left(|u|^2u\right)\right)}\pi_{n_2}(I_Nu)\overline{\pi_{n_3}(Iu)}\pi_{n_4}(Iu)\,dx\,dt\\
=&\int_0^t\int_{M}
I\left(|u|^2u\right)Iu\overline{Iu}Iudxdt-\int_0^t\int_{M}
I^2\left(|u|^2u\right)u\,\overline{u}\,u\,dx\,dt.\\
\end{align*}
Decomposing $u=\sum_{N}u_N$ into the sum of  Littlewood-Paley pieces as before and writing $I_2$ as
\begin{align}\label{term 2 dyadic}
I_2&=\sum_{N_i \in 2^{\Bbb{N}}}\Big(\int_0^t\int_{M}
I\big(u_{N_1}\overline{u_{N_2}}u_{N_3}\big)Iu_{N_4}\overline{Iu_{N_5}}Iu_{N_6}dxdt \notag\\
&\hspace{4ex}-\int_0^t\int_{M}
I^2\left(u_{N_1}\overline{u_{N_2}}u_{N_3}\right)u_{N_4}\overline{u_{N_5}}u_{N_6}dxdt\Big).
\end{align}
Let us denote by $N_{max}$ the largest of all $N_1,\ldots,N_6$ and $N_{med}$ the second one. We renumber $N_{j}$ for which delete $N_{max},N_{med}$ from it as $\tilde{N}_{j}$($\{\tilde{N}_1,\cdots,\tilde{N}_4\}=\{N_1,\cdots,N_6\}\backslash\{N_{max},N_{med}\}$). We will treat three cases below:
\begin{enumerate}
\item \textbf{Case one}: $N\gg N_{max}$ or $N_{max}\gg N_{med}$,
\item \textbf{Case two}: $N_{max}\gtrsim N\quad N_{max}\sim N_{med}$,  
\item \textbf{Case three}: $N_{max}\gtrsim N\quad N_{max}\sim N_{med}$.
\end{enumerate}

\noindent\textbf{Case One} : $N\gg N_{max}$ or $N_{max}\gg N_{med}$.

We give detailed proof for the case $N\gg N_{max}$ and similar argument is avaliable to $N_{max}\gg N_{med}$. Writing
\begin{equation}\label{high-low-separation}
u_{N_1}\overline{u_{N_3}}u_{N_3}=P_{\leq N}\left(u_{N_1}\overline{u_{N_3}}u_{N_3}\right)+P_{>N}\left(u_{N_1}\overline{u_{N_3}}u_{N_3}\right)
\end{equation}
and	substituting $\eqref{high-low-separation}$ for $\eqref{term 2 dyadic}$, we obtain
\begin{align}\label{high-low-term2}
\operatorname{Term}_2&=\sum_{N_i \in 2^{\N}}\left(\int_0^t\int_{M}
I(P_{\le N}+P_{>N})\left(u_{N_1}\overline{u_{N_2}}u_{N_3}\right)Iu_{N_4}\overline{Iu_{N_5}}Iu_{N_6}\,dx\,dt \right.\nonumber\\
&\hspace{3ex}\left .-\int_0^t\int_{M}
I^2(P_{\le N}+P_{>N})\left(u_{N_1}\overline{u_{N_2}}u_{N_3}\right)u_{N_4}\overline{u_{N_5}}u_{N_6}\,dx\,dt\right).\nonumber\\		
&=\sum_{N_i \in 2^{\N}}\bigg(\int_0^t\int_{M}
[(IP_{\le N}-I^{2}P_{\le N})+IP_{>N}-I^{2}P_{>N}]\left(u_{N_1}\overline{u_{N_2}}u_{N_3}\right)\notag\\
&\hspace{3ex}\times Iu_{N_4}\overline{Iu_{N_5}}Iu_{N_6}\,dx\,dt \bigg),
\end{align}
under the assumption that $N_{max}\ll N$. Noticing that $I_N,P_{>N}$ is self-adjoint and $L^{2}$ bounded, the almost orthogonal estimate $\eqref{almost}$ implies the following for every $q>0$
\begin{align*}
\|IP_{>N}(u_{N_1}\overline{u_{N_2}}u_{N_3})\|_{L^{2}_{x}}=&\sup\limits_{\|u_{N_{0}}\|_{L^{2}_{x}}=1}\int_{M}\overline{u_{N_{0}}}IP_{>N}(u_{N_1}\overline{u_{N_2}}u_{N_3})dx\\
=&\sup\limits_{\|u_{N_{0}}\|_{L^{2}_{x}}=1}\int_{M}I(\overline{u_{N_{0}}})P_{>N}(u_{N_1}\overline{u_{N_2}}u_{N_3})dx\\
\lesssim&\sup\limits_{\|u_{N_{0}}\|_{L^{2}_{x}}=1}\left|\int_{M}P_{>N}(I\overline{u_{N_{0}}})(u_{N_1}\overline{u_{N_2}}u_{N_3})dx\right|\\
\lesssim&N^{-q}\prod_{i=1}^{3}\|u_{N_{i}}\|_{L^{2}(M)}.
\end{align*}
Using the fact that $I^2 u_{N_i}=I u_{N_i}=u_{N_i}$ if ${N_i}\ll N$ and  the H\"{o}lder inequality and the   Strichartz estimate, then the contribution of the term for which $N_{max}\ll N$ to high frequency part of $\eqref{high-low-term2}$ is
\begin{align*}
&\sum_{N_{i}\ll N}\int_0^t\int_{M}
IP_{>N}\left(u_{N_1}\overline{u_{N_2}}u_{N_3}\right) Iu_{N_4}\overline{Iu_{N_5}}Iu_{N_6}\,dx\,dt\\
&\lesssim\sum_{N_{i}\ll N}\int_0^t\|IP_{>N}(u_{N_1}\overline{u_{N_2}}u_{N_3})\|_{L_{x}^{2}(M)}\|I_Nu_{N_4}\overline{Iu_{N_5}}Iu_{N_6}\|_{L^{2}_{x}(M)}\,dt\\
&\lesssim\|I_NP_{>N}(u_{N_1}\overline{u_{N_2}}u_{N_3})\|_{L_{t}^\infty L_x^2([0,\delta]\times M)}\prod_{i=4}^{6}\|u_{N_{i}}\|_{L_{t}^{3}L_x^6([0,\delta]\times M)}\\
&\lesssim N^{-q}\prod_{j=1}^{3}\Vert u_{N_i}\Vert_{X^{0,\frac{1}{2}+}([0,\delta]\times M)}\prod_{i=4}^{6}\|u_{N_{i}}\|_{L_{t}^{3}L_x^6([0,\delta]\times M)}\\
&\lesssim N^{-(q-\frac{2}{3})}\prod_{i=1}^{6}\|u_{N_{i}}\|_{L^{2}_{x}}.
\end{align*}
For  large $q_0=q-\frac{3}{2}$, we can complete the proof. In the last inequality, we use the following estimate which holds true for any compact manifold without boundary, 
\begin{align*}
  \Vert e^{it\Delta}P_Nf\Vert_{L_t^{3}L_x^6([0,\delta]\times M)}\lesssim N^{\frac{1}{3}}\Vert e^{it\Delta}P_Nf\Vert_{L_t^3L_x^{\frac{18}{5}}}\lesssim N^{\frac{2}{3}}\Vert P_Nf\Vert_{L^2(M)}.
\end{align*}
 
 Similar argument gives the same bound with $N^ {-q}\prod_{i=1}^{6}\|u_{N_{i}}\|_{L^{2}_{x}}$ for the term $$\sum_{N_{i}\ll N}\int_0^t\int_{M}
I^{2}P_{>N}\left(u_{N_1}\overline{u_{N_2}}u_{N_3}\right) Iu_{N_4}\overline{Iu_{N_5}}Iu_{N_6}\,dx\,dt.$$
Meanwhile the contribution of the case for which $N_{max}\ll N$ to low frequency part of (\ref{high-low-term2}) vanishes due to $I\left(P_{\leq N}u_{N_1}\overline{u_{N_3}}u_{N_3}\right)=I^2\left(P_{\leq N}u_{N_1}\overline{u_{N_3}}u_{N_3}\right)$.

A similar argument using Lemma \ref{orthogonal} shows that the contribution of $N_{max}\gg N_{med}$ is also harmless so we restrict attention to the case when $N_{max}\gtrsim N$ and $N_{max}\sim N_{med}$.

\noindent\textbf{Case two}: $N_{max}\gtrsim N,\hspace{1ex} N_{max}\sim N_{med}$.

We notice that $I$ is bounded as an $L^p$ multiplier, one can estimate the case of $N_{max}\gtrsim N$ via H\"older's inequality as follows:
\begin{align*}
I_2\lesssim& \sum_{\substack{N_{max}, N_{med}\gtrsim N,\\  \tilde N_i\leq N_{max}}}||u_{N_{max}}||_{L_{t,x}^{4+}}||u_{N_{med}}||_{L^{4+}_{t,x}}\prod_{i=1}^4||u_{\tilde N_i}||_{L_{t,x}^{8-}}.
\end{align*}

For the Zoll manifold, we  treat $\operatorname{Term}_2$ by making use of  $\eqref{L^p-zoll}$,
\begin{align*}
I_2&\lesssim \sum_{\substack{N_{max}, N_{med}\gtrsim N,\\ \tilde N_i\leq N_{max}}} (\langle N_{max}\rangle\langle N_{med}\rangle)^{\frac{1}{4}+}||u_{N_{max}}||_{X^{0,1/2+}}||u_{N_{med}}||_{X^{0,1/2+}}\prod_{i=1}^4\langle \tilde N_i\rangle^{\frac{7}{8}-}||u_{\tilde N_i}||_{X^{0,1/2+}}\\
&\lesssim \frac{1}{m(N_{max})N_{max}^{\frac{3}{4}-}m(N_{med})N_{med}^{\frac{3}{4}-}}\Vert Iu_{max}\Vert_{X^{1,\frac{1}{2}+}}\Vert Iu_{med}\Vert_{X^{1,\frac{1}{2}+}}\prod_{j=1}^{4}\frac{1}{\tilde N_i^{\frac{1}{8}+}}\Vert Iu_{\tilde{N}_j}\Vert_{X^{1,\frac{1}{2}+}}.
\end{align*}
For any $k>0$ and $\alpha>1-s$, we have  that
\begin{equation}\label{index}
m(k)k^{\alpha}\gtrsim\left\{
\begin{aligned}
	&1,\quad k\le N,\\
	&N^{\alpha},\quad k\ge 2N.\\
\end{aligned}
\right
.
\end{equation}
Applying the estimate above with $\alpha=\frac{3}{4}->\frac{5-\sqrt{21}}{4}>1-s$ and $\alpha=\frac{1}{8}+>\frac{5-\sqrt{21}}{4}$, then we have
\begin{align*}
 I_2&\lesssim \frac{1}{N^{\frac{3}{2}-}}\sum_{N_{max}\in2^{\Bbb{N},N_{med},\tilde{N}_i\lesssim N_{max}}}N_{max}^{0-}\Vert Iu_{max}\Vert_{X^{1,\frac{1}{2}+}}\Vert Iu_{med}\Vert_{X^{1,\frac{1}{2}+}}\prod_{i=1}^{4}\Vert Iu_{\tilde{N}_j}\Vert_{X^{1,\frac{1}{2}+}}\\
  &\lesssim\frac{1}{N^{\frac{3}{2}-}}\Vert Iu\Vert_{X^{1,\frac{1}{2}+}}. 
\end{align*}
 Hence, we complete the proof of $\operatorname{Term_2}$ for Zoll manifold.
 
 When $M=\Bbb{S}^2\times\Bbb{S}^1$, we  treat $\operatorname{Term}_2$ by making use of  $\eqref{L^p-pro}$,
\begin{align*}
I_2&\lesssim \sum_{\substack{N_{max}, N_{med}\gtrsim N,\\ \tilde N_i\leq N_{max}}} (\langle N_{max}\rangle\langle N_{med}\rangle)^{\frac{3}{8}+}||u_{N_{max}}||_{X^{0,1/2+}}||u_{N_{med}}||_{X^{0,1/2+}}\prod_{i=1}^4\langle \tilde N_i\rangle^{\frac{7}{8}-}||u_{\tilde N_i}||_{X^{0,1/2+}}\\
&\lesssim \frac{1}{m(N_{max})N_{max}^{\frac{5}{8}-}m(N_{med})N_{med}^{\frac{5}{8}-}}\Vert Iu_{max}\Vert_{X^{1,\frac{1}{2}+}}\Vert Iu_{med}\Vert_{X^{1,\frac{1}{2}+}}\prod_{j=1}^{4}\frac{1}{\tilde N_i^{\frac{1}{8}+}}\Vert Iu_{\tilde{N}_j}\Vert_{X^{1,\frac{1}{2}+}}.
\end{align*}
For any $k>0$ and $\alpha>1-s$, we can check that
\begin{equation}\label{index}
m(k)k^{\alpha}\gtrsim\left\{
\begin{aligned}
	&1,\quad k\le N,\\
	&N^{\alpha},\quad k\ge 2N.\\
\end{aligned}
\right
.
\end{equation}
Applying the estimate above with $\alpha=\frac{5}{8}->\frac{5-\sqrt{21}}{4}>1-s$ and $\alpha=\frac{1}{8}+>\frac{5-\sqrt{21}}{4}$, then we have
\begin{align*}
 I_2&\lesssim \frac{1}{N^{\frac{5}{4}-}}\sum_{N_{max}\in2^{\Bbb{N},N_{med},\tilde{N}_i\lesssim N_{max}}}N_{max}^{0-}\Vert Iu_{max}\Vert_{X^{1,\frac{1}{2}+}}\Vert Iu_{med}\Vert_{X^{1,\frac{1}{2}+}}\prod_{i=1}^{4}\Vert Iu_{\tilde{N}_j}\Vert_{X^{1,\frac{1}{2}+}}\\
  &\lesssim\frac{1}{N^{\frac{5}{4}-}}\Vert Iu\Vert_{X^{1,\frac{1}{2}+}}. 
\end{align*}
 Hence, we complete the proof of Proposition \ref{prop}.
\endproof
\subsection{Polynomial bounds on $\Vert u(T)\Vert_{H^s(M)}$ and global well-posedness:}\label{polybounds}

In this subsection, we prove the global well-posedness and give polynomial bound of $H^{s}$ norm to the solution $u$. By the definition of modified energy and the Gagliardo-Nirenberg inequality, we obtain
\begin{align*}
  E(Iu_0)=&\frac{1}{2}\Vert Iu_0\Vert_{\dot{H}^1(M)}^2+\frac{1}{4}\Vert Iu_0\Vert_{L^4(M)}^4\\
  \lesssim& \Vert\nabla Iu_0\Vert_{L^2(M)}^2+\Vert\nabla Iu_0\Vert_{L^2(M)}^3\Vert Iu_0\Vert_{L^2(M)}\lesssim N^{3(1-s)},
\end{align*}
since $\Vert Iu_0\Vert_{\dot{H}^1}\lesssim N^{1-s}$.

\noindent\emph{\textbf{The case of Zoll manifold}}. \hspace{1ex}
The energy increment of $\eqref{I}$ on Zoll manifold obtained in the previous section becomes 
\begin{align}\label{energy-in}
  E(Iu(\delta))&\lesssim E(Iu(0))+N^{-1+}\Vert Iu_0\Vert_{H^1}^4+N^{-\frac{1}{2}+}\delta^{\frac{1}{4}-}\Vert Iu_0\Vert_{H^1}^4+N^{-\frac{3}{2}+}\Vert Iu_0\Vert_{H^1}^6\notag\\
  &\lesssim N^{3(1-s)}+ N^{-1+}N^{4(1-s)}+N^{-\frac{1}{2}+}N^{-\frac{(1-s)}{2s-1}}N^{4(1-s)}+N^{-\frac{3}{2}+}N^{6(1-s)}.
\end{align}
In order to encounter the time $T$ we need  to iterate $\frac{T}{\delta}$ times, 
\begin{align}\label{iterate-step}
  \frac{T}{\delta}\sim TN^{\frac{4(1-s)}{2s-1}+}.
\end{align}
Combining $\eqref{energy-in}$ and $\eqref{iterate-step}$, we have
\begin{align*}
  E(Iu(T))&\lesssim E(Iu_0)+\frac{T}{\delta}\big(N^{-1+}N^{4(1-s)}+N^{-\frac{1}{2}+}N^{-\frac{(1-s)}{2s-1}}N^{4(1-s)}+N^{-\frac{3}{2}+}N^{6(1-s)}\big)\\
  &\lesssim N^{3(1-s)}+TN^{\frac{4(1-s)}{2s-1}+}\big(N^{-1+}N^{4(1-s)}+N^{-\frac{1}{2}+}N^{-\frac{(1-s)}{2s-1}+}N^{4(1-s)}+N^{-\frac{3}{2}+}N^{6(1-s)}\big).
\end{align*}
To keep on the iteration at each step $[k\delta,(k+1)\delta]$, we need the following bounds
\begin{gather*}
  TN^{\frac{4(1-s)}{2s-1}+}N^{-\frac{(1-s)}{2s-1}+}N^{-\frac{1}{2}+}N^{4(1-s)}\lesssim N^{3(1-s)},\\
   TN^{\frac{4(1-s)}{2s-1}+}N^{-1+}N^{4(1-s)}\lesssim N^{3(1-s)},\\
    TN^{\frac{4(1-s)}{2s-1}+}N^{-\frac{3}{2}+}N^{6(1-s)}\lesssim N^{3(1-s)}
\end{gather*}
which gives the constraint of $s$ as
\begin{gather*}
  \frac{3(1-s)}{2s-1}+<s-\frac{1}{2}\Longrightarrow s>\frac{\sqrt{21}-1}{4},\\
  \frac{4(1-s)}{2s-1}+<s\Longrightarrow s>\frac{\sqrt{41}-3}{4},
  \frac{4(1-s)}{2s-1}+<3s-\frac{3}{2}\Longrightarrow s>\frac{5}{6}.
\end{gather*}
Since $\frac{\sqrt{21}-1}{4}>\frac{\sqrt{41}-3}{4}>\frac{5}{6}$, we choose $s\in(\frac{\sqrt{21}-1}{4},1)$. We also get that 
\begin{align*}
  T\sim \min\{N^{-\frac{3(1-s)}{2s-1}+s-\frac{1}{2}-},N^{-\frac{4(1-s)}{2s-1}+s-},N^{-\frac{4(1-s)}{2s-1}+3s-\frac{3}{2}-}\}\lesssim N^{-\frac{3(1-s)}{2s-1}+s-\frac{1}{2}-}. 
\end{align*} We set $p(Zoll)=\frac{4(1-s)}{2s-1}+3s-\frac{3}{2}$ for simpleness. As a direct consequence, we have
\begin{align*}
  \Vert u(T)\Vert_{H^s(M)}\lesssim\Vert Iu(T)\Vert_{H^1(M)}\lesssim E^\frac{1}{2}(Iu(T))\sim N^{\frac{3}{2}(1-s)}\lesssim N^{\frac{3}{2p(Zoll)}(1-s)+}.
\end{align*}

\noindent\emph{\textbf{Product of spheres $\Bbb{S}^2\times\Bbb{S}^1$}}.
The energy increment of $\eqref{I}$ on Zoll manifold obtained in the previous section becomes 
\begin{align}\label{energy-in}
  E(Iu(\delta))&\lesssim E(Iu(0))+N^{-\frac{1}{2}+}\Vert Iu_0\Vert_{H^1}^4+N^{-\frac{1}{4}+}\delta^{\frac{1}{6}-}\Vert Iu_0\Vert_{H^1}^4+N^{-\frac{5}{4}+}\Vert Iu_0\Vert_{H^1}^6\notag\\
  &\lesssim N^{3(1-s)}+ N^{-\frac{1}{2}+}N^{4(1-s)}+N^{-\frac{1}{4}+}N^{-\frac{(1-s)}{4s-3}}N^{4(1-s)}+N^{-\frac{5}{4}+}N^{6(1-s)}.
\end{align}
In order to encounter the time $T$ we need  to iterate $\frac{T}{\delta}$ times, 
\begin{align}\label{iterate-step}
  \frac{T}{\delta}\sim TN^{\frac{6(1-s)}{4s-3}+}.
\end{align}
Combining $\eqref{energy-in}$ and $\eqref{iterate-step}$, we have
\begin{align*}
  E(Iu(T))&\lesssim E(Iu_0)+\frac{T}{\delta}\big(N^{-\frac{1}{2}+}N^{4(1-s)}+N^{-\frac{1}{4}+}N^{-\frac{(1-s)}{4s-3}}N^{4(1-s)}+N^{-\frac{5}{4}+}N^{6(1-s)}\big)\\
  &\lesssim N^{3(1-s)}+TN^{\frac{6(1-s)}{4s-3}+}\big(N^{-\frac{1}{2}+}N^{4(1-s)}+N^{-\frac{1}{4}+}N^{-\frac{(1-s)}{2s-1}+}N^{4(1-s)}+N^{-\frac{5}{4}+}N^{6(1-s)}\big).
\end{align*}
To keep on the iteration at each step $[k\delta,(k+1)\delta]$, we need the following bounds
\begin{gather*}
  TN^{\frac{6(1-s)}{2s-1}+}N^{-\frac{(1-s)}{4s-3}}N^{-\frac{1}{4}+}N^{4(1-s)}\lesssim N^{3(1-s)}\\
   TN^{\frac{6(1-s)}{4s-3}+}N^{-\frac{1}{2}+}N^{4(1-s)}\lesssim N^{3(1-s)}\\
    TN^{\frac{6(1-s)}{4s-3}+}N^{-\frac{5}{4}+}N^{6(1-s)}\lesssim N^{3(1-s)}
\end{gather*}
which gives the constraint of $s$ as
\begin{gather*}
  \frac{5(1-s)}{4s-3}+<s-\frac{3}{4}\Longrightarrow s>\frac{1+3\sqrt{5}}{8}\approx0.963,\\
  \frac{6(1-s)}{4s-3}+<s-\frac{1}{2}\Longrightarrow s>\frac{\sqrt{73}-1}{8}\approx0.943,\\
  \frac{6(1-s)}{4s-3}+<3s-\frac{3}{2}\Longrightarrow s>\frac{5+\sqrt{34}}{12}\approx0.902.
\end{gather*}
Therefore, we choose $s\in(\frac{1+\sqrt{5}}{8},1)$. We also get that 
\begin{align*}
  T\sim \min\{N^{-\frac{5(1-s)}{4s-3}+s-\frac{3}{4}-},N^{-\frac{6(1-s)}{4s-3}+s-\frac{1}{2}-},N^{-\frac{6(1-s)}{4s-3}+3s-\frac{7}{4}-}\}\lesssim N^{-\frac{5(1-s)}{4s-3}+s-\frac{3}{4}}. 
\end{align*} We set $p(\Bbb{S}^2\times\Bbb{S}^1)=-\frac{5(1-s)}{4s-3}+s-\frac{3}{4}$ for simpleness. As a direct consequence, we have
\begin{align*}
  \Vert u(T)\Vert_{H^s(M)}\lesssim\Vert Iu(T)\Vert_{H^1(M)}\lesssim E^\frac{1}{2}(Iu(T))\sim N^{\frac{3}{2}(1-s)}\lesssim N^{\frac{3}{2p(\Bbb{S}^2\times\Bbb{S}^1)}(1-s)+}.
\end{align*}
Hence, we complete the proof of Theorem \ref{Thm1}.\endproof




\begin{center}

\end{center}

\end{document}